%% file: main.tex
\documentclass[final,1p,times]{elsarticle}
\usepackage{amssymb}
\usepackage{amsmath}
\usepackage{amsthm}
\usepackage{amsfonts}
\usepackage{graphicx}
\usepackage{booktabs}
\usepackage{array}
\usepackage{multirow}
\usepackage{longtable}
\usepackage{natbib}
\usepackage{url}
\usepackage{algorithm}
\usepackage{algpseudocode}
\usepackage{hyperref}
\newtheorem{theorem}{Theorem}
\newtheorem{lemma}{Lemma}
\newtheorem{proposition}{Proposition}
\newtheorem{corollary}{Corollary}
\newtheorem{assumption}{Assumption}
\newtheorem{definition}{Definition}
\theoremstyle{remark}
\newtheorem{remark}{Remark}
\journal{Computers \& Industrial Engineering}
\begin{document}
\begin{frontmatter}
\title{Coupling-Aware Sales and Operations Planning with Forced Co-Production in Polymer Production Plants}

\affiliation[kfupm]{organization={Industrial and Systems Engineering (ISE) Department and IRC for Smart Mobility and Logistics (IRC-SML), \\King Fahd University of Petroleum and Minerals (KFUPM)},
            city={Dhahran},
            country={Saudi Arabia}}
\affiliation[its1]{organization={Industrial and Systems Engineering Department, Sepuluh Nopember Institute of Technology},
            city={Surabaya},
            country={Indonesia}}
\affiliation[its2]{organization={Business Statistics Department, Sepuluh Nopember Institute of Technology},
            city={Surabaya},
            country={Indonesia}}
\author[kfupm]{Mansur M. Arief}
\author[its1]{Yan Akhra Pratama}
\author[its1]{Bilal Ahmadi}
\author[its2]{Apsarini Pradipta}
\author[its1]{Iwan Vanany}



\begin{abstract}
Profitability metrics, such as the single-product gross profit margin, constitute the standard toolbox to prioritize products in sales and operations planning (S\&OP). These metrics leverage the premise that production lines can be committed independently, which endows them with a simple per-ton ranking that a planner can compute and use easily in practice. However, coupled plants, especially polymer plants whose parallel lines draw simultaneously on a single bulk feed, can fundamentally undermine this premise and mislead to value-destroying production plans. We propose a framework built on what we call the Augmented Gross Profit for Product Clusters (AGPPC), which converts the per-ton ranking metric into a per-coupled-hour ranking that carries a per-instance optimality certificate. We present the theory of AGPPC, which combines the co-production column concept with a fluid relaxation of the planning problem, alongside an integrated bilinear mixed integer program of the plant and a McCormick-linearized baseline that provides certified bounds for it, and we demonstrate its effectiveness on real operational data from an Indonesian polymer producer, where the coupling-aware metric alone raises operating profit by 7.6\% and the full optimization by 28\% over current practice.
\end{abstract}
\end{frontmatter}

\section{Introduction}
\label{sec1}
Sales and operations planning (S\&OP) has constituted a powerful toolbox to reconcile commercial commitments with plant capability in process industries \citep{gnanendran2026practical, maravelias2021chemical}. From a planning standpoint, the difficulty is that capacity is scarce relative to the product portfolio, so that some criterion must rank products before any schedule can be built. Economically, this difficulty manifests itself as the need for a per-unit valuation of the bottleneck, and prioritization rules, cast under the umbrella of so-called profitability metrics, have been substantially used over the years to supply that valuation. These rules range from the gross profit margin per ton to contribution per machine hour \citep{pochet2006production, harjunkoski2014scope}. In many settings they can be shown to reproduce the optimal ordering exactly, thus forming a workhorse for planning problems that would otherwise be combinatorial.

While powerful, it is also well known that, to endow a prioritization rule with the property that ranking by it is optimal, one often requires knowledge and careful analysis of the underlying plant structure \citep{shah2005process, silva2023optimization}. This is particularly so in process industries, because a criterion that gives substantial simplification, such as the margin ranking, can also greatly harm performance if it is not configured to suit the considered plant. The high-level issue is that ranking products by their own economics alone does not necessarily produce a better plan than an arbitrary order. To be economically valid, the object being ranked must coincide with what the plant can actually commit to which, in turn, may be a bundle of products rather than one. On the other hand, general-purpose methods, such as solving the full mixed integer program (MIP), are advantageously exact and require no structural argument, but with the price that they are opaque at the planning desk and are not always available there.

Our goal in this paper is to study an approach that gives economic guarantees for product prioritization in plants whose lines are coupled. By ``coupled'' here we mean that parallel lines draw simultaneously on a shared bulk feed under grade compatibility restrictions, so that committing the primary line to a grade forces co-production on the remaining lines at rates whose total is fixed by the feed. On the other hand, the commercial and technical data of the plant (prices, bills of material, rate envelopes, compatibility matrices, demand forecasts) are known and can be engineered, so that we can evaluate candidate commitments and rank them. As will be revealed in the later sections, we would also need to place certain conditions on the demand state to elicit our guarantees. Thus, our terminology of ``coupled'' is not meant to describe an intractable plant, but instead meant to highlight the relaxation of the traditional requirement of line independence in using a more coupling-aware prioritization criterion.

To this end, while one may attempt to use the criterion already in place, we show that, when applied in such coupled settings, the margin ranking can destroy value systematically if used arbitrarily in coupled plants. We highlight that this is an issue that goes beyond theoretical interest. By not bearing an economic guarantee, the loss is hidden from the planner until it is too late, hus potentially trapping the planner into making a large mistake without knowing it. Our approach, in some sense, aims to convert this familiar criterion that lacks guarantees and risks undetected early loss into one that carries an optimality certificate under some more general coupled-plant settings.

Our motivation comes from a petrochemical-derived polymer facility in Indonesia whose three continuous lines share one bulk feed stream. Process $k{=}1$ (P1) is the primary pelletizing line, $k{=}2$ (P2) a secondary pelletizing line, and $k{=}3$ (P3) a granulation line. Feed arrives under a take-or-pay contract at an essentially constant rate, and on-site storage for it is tight, so the plant must convert roughly the contracted tonnage every month whether or not orders exist. Compatibility is grade-specific: running grade $i$ on P1 admits only a small set of grades on P2 and on P3, and the instantaneous rates of the three lines must sum to the feed rate. Figure~\ref{fig:interprocess_compat} gives a pictorial overview of the interaction.

\begin{figure}[htbp]
    \centering
    \includegraphics[width=1.0\linewidth]{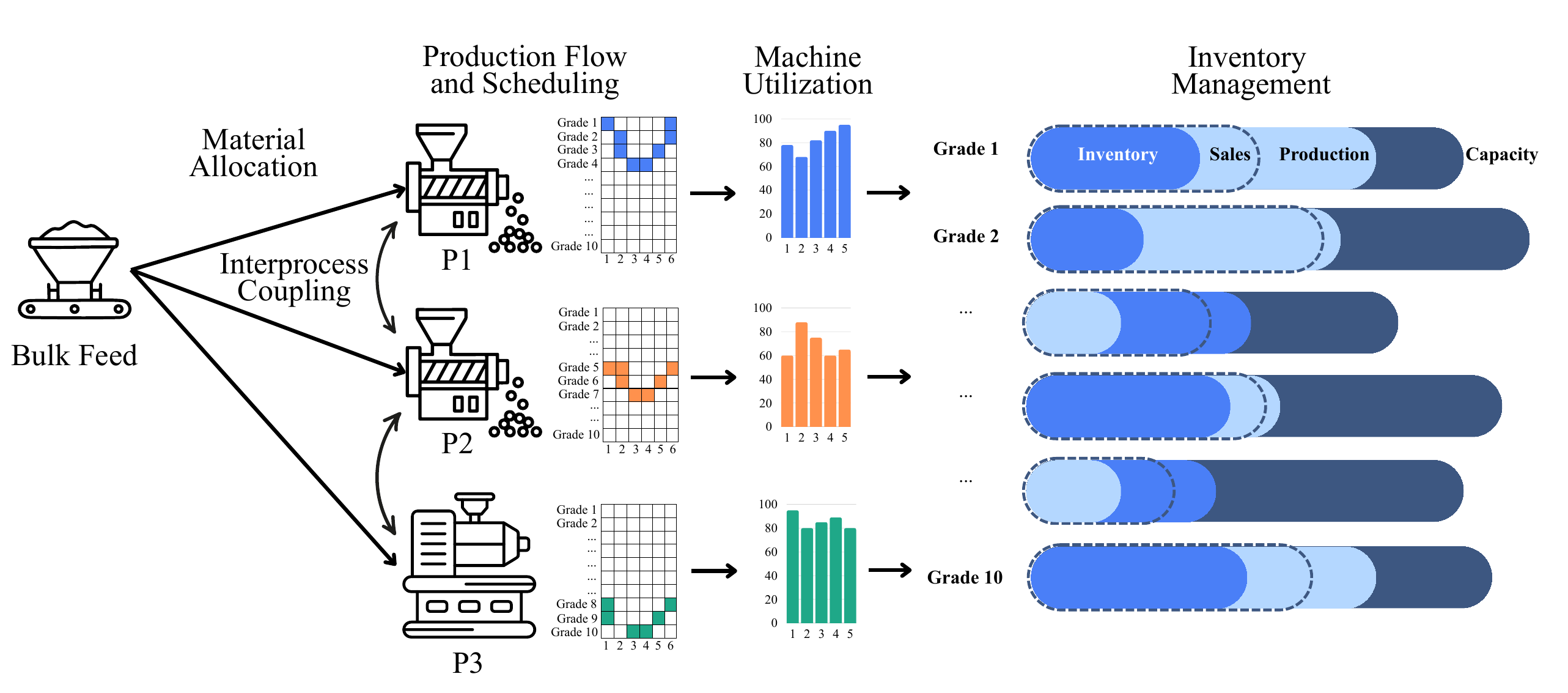}
    \caption{Overview of the coupled planning problem. A take-or-pay bulk feed stream ($\rho_t \approx 30.5$ tph) is allocated simultaneously to the primary pelletizing line P1, the secondary pelletizing line P2, and the granulation line P3; the arrows between the lines mark the interprocess coupling, i.e., the anchor grade selected on P1 restricts P2 to the compatibility set $\mathcal{C}_2(i)$ and P3 to $\mathcal{C}_3(i)$, and the three instantaneous rates must sum to the feed rate. Each line then carries a grade-by-period production plan, an operating-hours profile against its monthly capacity, and a per-grade balance of opening inventory, sales, production, and stock capacity.}
    \label{fig:interprocess_compat}
\end{figure}

Two characteristics of the plant system shape the entire planning problem. First, co-production is \emph{forced}, i.e., because the rate envelopes of P1 and P2 together cannot absorb the feed, the granulation line must run whenever the primary line runs, whether or not granule demand exists. Second, the \emph{economics} of grades are coupled: the value of committing the plant to a grade depends on the margins and remaining demands of all forced co-products, not on the margin of that grade alone. A grade with an attractive own margin can anchor a bundle whose forced co-products destroy value, while a modest grade can anchor a highly profitable one. Siloed coordination between the commercial and the operations functions is well documented to yield conflicting objectives and suboptimal outcomes \citep{oliva2011cross, das2017planning, hulthen2017challenges}, and in process industries the production constraints are tight enough \citep{zhang2007effective, kadambur2015multi} that sales planning has little room to move without producing infeasible plans. Planning practice at the facility, implemented in its planning workbook and, in our experience, common across the industry, nonetheless ranks grades by price minus direct cost per ton and fills the monthly calendar from the top of that ranking. In this paper we formalize why that is misleading particularly in a coupled-plant system, quantify it on real data, and propose a corrected metric with provable properties.

In this work, we aim to study an approach that obtains economically defensible plans for coupled plants by endowing a versatile ranking rule with guarantees. On a high level, our approach comprises a two-layer framework, where the first layer prices one hour of coupled plant operation, while the second layer allocates the scarce hour budget according to that price. To endow guarantees in this ``price-then-allocate'' approach, our framework utilizes three key ingredients:
\begin{enumerate}
\item \emph{Co-production columns}: we read off the compatibility and feed-tracking constraints the smallest unit of operation the plant actually permits, namely a compatible triple of grades together with a feasible rate mix, and we make that unit, rather than the individual grade, the object being ranked.
\item \emph{Demand-aware effective margins}: we discount the contribution of any forced co-product whose remaining demand is already exhausted to its negative direct cost, so that the value of a commitment reflects the current demand state rather than a static price sheet.
\item \emph{Certified allocation}: we allocate the coupled-hour budget greedily by the resulting metric, and we bound the loss of that allocation against a fluid relaxation of the planning problem that keeps precisely the coupling constraints and drops the rest.
\end{enumerate}
We call the metric that the first two ingredients produce the \emph{Augmented Gross Profit for Product Clusters} (AGPPC). The relaxation underlying the third ingredient is a linear program, so its optimum is computable on any instance, which is what builds our certificate. Our approach is robust in the sense that it recovers the optimum of the relaxation exactly when demands are rich relative to the coupled-hour budget, while it still returns a valid and computable bound on its own suboptimality when they are not. Around this framework we build an integrated bilinear mixed integer program of the plant, which prices the side constraints and the multi-period stock decisions that the metric deliberately ignores, and a McCormick-linearized baseline that bounds the exact model from above and, with a repair step, from below. To our best knowledge, a coupling-aware profitability metric with this type of guarantee is the first of its kind in the S\&OP literature.

In the following, Section~\ref{sec:lit} positions our work in the literature. Section~\ref{sec:model} describes the plant and presents the integrated formulation, closing with the limitation of margin-based prioritization that motivates everything after it. Section~\ref{sec:agppc} presents in detail our AGPPC framework, its underpinning mechanism, and its guarantees. Section~\ref{sec:method} describes the solution methodology, including the linearized baseline and the strategy variants. Then, Section~\ref{sec:experiments} demonstrates and compares our framework with existing practice in several experiments on real industrial data, including the deployed planning tool. Section~\ref{sec:conclusion} discusses limitations and future work.

\section{Related Work and Positioning}
\label{sec:lit}
\subsection{Planning and Compatibility in Process Industries}
Classic lot-sizing models \citep{pochet2006production} and parallel-line lot-sizing formulations with sequence-dependent setups \citep{alimian2022parallel} laid the foundation for MIP-based decision support in continuous and semi-continuous systems, and have been extended to parallel machines \citep{gupta2006flowshop, meyr2013decomposition}, campaign planning \citep{grunow2002campaign}, and rolling-horizon policies \citep{alimian2022parallel}. In petrochemicals specifically, \citet{alfares2002optimization} maximize profit through process selection and capacity allocation, \citet{mrad2016optimum} extend multi-grade planning to multiple periods, plants, and suppliers, and \citet{sidki2025monolithic} co-optimize production with pipeline distribution at industrial scale. Most of these formulations, however, assume that any product can be made on any line, which rarely holds in polymer processing because of contamination and grade compatibility limits \citep{elyasi2024imperialist, du2024plant}.

Compatibility itself has been modeled through binary feasibility matrices, directed graphs, and transition cost functions. Food and pharmaceutical applications handle product-to-line compatibility through assignment matrices or sequence-dependent setups \citep{akkerman2010food, mendez2006state}, with infeasible transitions either penalized \citep{pinedo2016scheduling} or prohibited outright by logic constraints \citep{raman1994modelling}. Polymer-specific work concentrates on transitions and sequencing: \citet{shi2016optimization} optimize grade-transition trajectories to minimize off-specification production, and \citet{elyasi2024imperialist} schedule unrelated parallel machines under machine-job compatibility. Process-design formulations do optimize equipment and material compatibility jointly \citep{pedrozo2021optimal}, and integrated production-distribution models couple planning with downstream flow \citep{lohmer2021production}, but only a handful of papers integrate grade compatibility into high-level planning \citep{younes2013quality, ghaleb2023dynamic}, and these are tailored to a single line configuration. What is consistently absent is the coupling of compatibility with \emph{upstream bulk allocation}, which is precisely the mechanism that makes co-production forced in our setting.

\subsection{Strategy, Bilinear Structure, and Positioning}
Hybrid make-to-order and make-to-stock positioning is well studied for decoupled discrete production \citep{olhager2003strategic, soman2004combined, rafiei2011order, peeters2020hybrid}, but not for continuous coupled lines where the strategy space itself is restricted by compatibility. On the computational side, the rate times hour equals quantity relationship that arises whenever both the operating rate and the operating duration are decisions renders planning models bilinear and hence nonconvex. The standard convexification is the McCormick envelope \citep{mccormick1976computability, alkhayyal1983jointly}, whose tightness depends on the bounds of the participating variables and which underpins global optimization of bilinear process networks \citep{quesada1995global}; piecewise linearization is a common alternative at the cost of additional binaries \citep{diabat2015location}. Modern global solvers handle such constraints directly through spatial branch-and-bound \citep{sahinidis1996baron, gurobi2024, bolusani2024scip}. We exploit both routes: the exact bilinear model solved by spatial branch-and-bound, and a bound-based linearization used as a fast baseline and bound provider.

Table~\ref{tab:literature_summary} summarizes representative prior work. We identify three gaps. First, an \emph{integration gap}: compatibility and bulk allocation are treated separately, so their interaction, which drives coupled co-production, is never fully captured analyzed systematically. Second, a \emph{profitability-metric gap}: planning approaches and industrial practice alike evaluate products by single-product margins even when compatibility couples the economics of entire bundles, and no prior work supplies a corrected metric with guarantees. Third, a \emph{strategy gap}: hybrid positioning has not been studied for coupled continuous lines. This work addresses all three.

\begin{table}[htbp]
\centering
\caption{Prior studies on sales and operations planning with compatibility modeling.}
\label{tab:literature_summary}
\resizebox{\textwidth}{!}{%
\scriptsize
\begin{tabular}{@{}p{2.7cm}p{2.5cm}p{2.5cm}p{1.7cm}p{1.3cm}@{}}
\toprule
\textbf{Study} & \textbf{Domain} & \textbf{Method} & \textbf{Compatibility} & \textbf{Bulk Feed} \\
\midrule
\citet{pochet2006production} & General lot-sizing & MIP & No & Yes\\
\citet{alimian2022parallel} & Multi-line manufacturing& MIP + rolling horizon & Limited (setups) & No\\
\citet{akkerman2010food} & Food production & Heuristics/MIP & Yes & No \\
\citet{mendez2006state} & Pharma Production & Scheduling models & Yes & No\\
\citet{shi2016optimization} & Polyethylene grade transitions & Dynamic optimization (NLP) & Transitions only & No \\
\citet{elyasi2024imperialist} & Multi-line scheduling & Metaheuristic & Yes & No\\
\citet{ghaleb2023dynamic} & Thermoplastics & Simulation + dispatching rules & Limited & No \\
\citet{mrad2016optimum} & Petrochemicals (multi-grade) & Multi-period MILP & Limited (transition waste) & Yes \\
\citet{sidki2025monolithic} & Refinery logistics & Monolithic MILP & Limited & Yes \\
\citet{peeters2020hybrid} & Discrete manufacturing & Queueing/heuristics & No & No \\
\textbf{This work} & Polymer Plant & Bilinear MIP + provable heuristic & \textbf{Yes} & \textbf{Yes}\\
\bottomrule
\end{tabular}
}
\end{table}

\section{The Coupled Planning Problem and Its Model}
\label{sec:model}
We first describe the plant and fix notation (Sections~\ref{sec:problem}-\ref{sec:notation}). We then present the objective and the constraints, ordered so that the interprocess coupling, which gives the problem its character, appears last and can be read against everything it interacts with (Sections~\ref{sec:model_overview}-\ref{sec:compat}). With these in place, we describe, as the key message of this section, the peril of applying the prevailing profitability metric to such a plant (Section~\ref{sec:whymislead}).

\subsection{Problem Description}
\label{sec:problem}
We consider a plant that converts a bulk raw material into $|\mathcal{I}|$ finished grades on $|\mathcal{K}|$ parallel continuous processes over a horizon of $|\mathcal{T}|$ monthly periods. The processes draw simultaneously from a common feed stream delivered at an essentially constant rate $\rho_t$ (tons per hour). This means whenever P1 runs a grade, P2 and P3 must run \emph{compatible} grades such that the combined instantaneous consumption of the three lines tracks $\rho_t$. Running grade $i$ on P1 admits only the grades in a set $\mathcal{C}_2(i)$ on P2 and $\mathcal{C}_3(i)$ on P3.

Each period, the planner decides which grades to sell and in what quantities (bounded by demand forecasts), which grades to produce on which process, at what rate and for how many hours, which materials to purchase, and how much product and material inventory to carry. The feed is delivered under a take-or-pay contract ($S_t$ tons per month) with limited on-site storage, which effectively forces the plant to convert most of the delivered tonnage within the month. The objective is to maximize profit: sales revenue minus procurement, electricity, and bagging costs, with unmet demand penalized.

\subsection{Notation}
\label{sec:notation}
Tables~\ref{tab:params}-\ref{tab:vars} summarize the notation. Sets are $\mathcal{I} = \{1,\dots,I\}$ of product grades with index $i$, $\mathcal{J} = \{1,\dots,J\}$ of materials and utilities with index $j$ (with $\mathcal{J}^{RM} \subset \mathcal{J}$ the raw materials), $\mathcal{K} = \{1,\dots,K\}$ of processes with index $k$, and $\mathcal{T} = \{1,\dots,T\}$ of monthly periods with index $t$.

\begin{table}[htbp]
\centering\caption{Parameters.}\label{tab:params}
\small
\begin{tabular}{@{}ll@{}}
\toprule
Symbol & Description \\
\midrule
$d_{it}$, $p_{it}$ & demand forecast and unit price of grade $i$ in period $t$ (\$/ton) \\
$c_j$ & unit cost of material $j$; $c^E$ electricity price (\$/kWh) \\
$a_{ijk}$ & bill-of-material: units of $j$ per ton of grade $i$ on process $k$ \\
$\underline{r}_{ik}, \bar{r}_{ik}$ & minimum / maximum production rate of grade $i$ on process $k$ (tph) \\
$\rho_t$ & bulk feed rate in period $t$ (tph); $S_t$ contracted bulk supply (tons) \\
$\bar{H}_k$ & maximum operating hours of process $k$ per period \\
$e_k$ & power consumption of process $k$ (kWh/h); $b_k$ bagging cost (\$/ton) \\
$\Phi^{12}_{ii'}, \Phi^{13}_{ii'} \in \{0,1\}$ & P1$\to$P2 / P1$\to$P3 compatibility indicators \\
$\delta^{\min}_{ii'}$ & minimum transition quantity from grade $i$ to $i'$ (tons) \\
$\phi_{ii'}, \alpha_{ii'} \in \{0,1\}$ & production-flow adjacency and grade-aggregation indicators \\
$\underline{q}_i$, $\bar{q}_{ik}$ & minimum production quantity of $i$; maximum quantity of $i$ on process $k$ (tons) \\
$\sigma^0_i, \iota^0_j$ & initial product stock / material inventory \\
$\underline{\sigma}_i, \bar{\sigma}_i$; $\underline{\iota}_j, \bar{\iota}_j$ & product stock and material inventory bounds \\
$M$, $\epsilon$ & big-M constant and numerical tolerance \\
\bottomrule
\end{tabular}
\end{table}

\begin{table}[htbp]
\centering\caption{Decision variables.}\label{tab:vars}
\small
\begin{tabular}{@{}ll@{}}
\toprule
Symbol & Description \\
\midrule
$x_{ikt} \geq 0$ & production quantity of grade $i$ on process $k$ in period $t$ (tons) \\
$y_{ikt} \in \{0,1\}$ & 1 if grade $i$ runs on process $k$ in period $t$ \\
$r_{ikt} \geq 0$, $h_{ikt} \geq 0$ & production rate (tph) and operating hours \\
$y^{tr}_{ii't} \in \{0,1\}$ & 1 if a transition from grade $i$ to $i'$ occurs in period $t$ \\
$z^{out}_{it}, z^{in}_{it} \geq 0$ & sales and production inflow of grade $i$ in period $t$ \\
$\sigma^{s}_{it}, \sigma^{e}_{it} \geq 0$ & starting and ending stock of grade $i$ \\
$u_{jt} \geq 0$, $v_{jt} \in \{0,1\}$ & purchase quantity and purchase indicator of material $j$ \\
$w_{jkt} \geq 0$ & usage of material $j$ by process $k$ \\
$\iota^{s}_{jt}, \iota^{e}_{jt} \geq 0$ & starting and ending inventory of material $j$ \\
\bottomrule
\end{tabular}
\end{table}

\subsection{Objective Function}
\label{sec:model_overview}
The formulation follows the physical and commercial flow of the plant. Feed and additives are purchased ($u_{jt}$), held ($\iota_{jt}$), and consumed by the lines ($w_{jkt}$) in proportion to what each line produces. Each line runs a grade ($y_{ikt}$) at a rate ($r_{ikt}$) for a number of hours ($h_{ikt}$), and the product of the two is the tonnage produced ($x_{ikt}$), which flows into stock ($\sigma_{it}$) from which sales ($z^{out}_{it}$) are served. The objective maximizes revenue minus procurement, electricity, and bagging costs, minus a per-ton penalty on unmet demand:
\begin{align}
\max \;\; & \sum_{i,t} p_{it} z^{out}_{it}
- \sum_{j,t} c_j u_{jt}
- \sum_{i,k,t} c^E e_k h_{ikt}
- \sum_{i,k,t} b_k x_{ikt}
- \lambda \sum_{i,t} \left( d_{it} - z^{out}_{it} \right),
\label{eq:objective}
\end{align}
with penalty weight $\lambda$. Materials are expensed at purchase, which is a cash view consistent with the take-or-pay contract; electricity is charged per operating hour and bagging per ton produced. In reporting we also track the \emph{usage-based operating profit}
\begin{equation}
    \Pi = \sum_{i,t} p_{it} z^{out}_{it} - \sum_{j,k,t} c_j w_{jkt} - \sum_{i,k,t} c^E e_k h_{ikt} - \sum_{i,k,t} b_k x_{ikt},
    \label{eq:profit}
\end{equation}
which expenses materials as they are consumed. Two reporting conventions therefore coexist. The strategy ranking is identical under \eqref{eq:objective} and under $\Pi$ (Table~\ref{tab:strategy}); the two differ in level only because \eqref{eq:objective} expenses the full take-or-pay purchase in the month of delivery while $\Pi$ expenses feed as it is converted. We report $\Pi$ throughout, because it is the only accounting on which the optimizer and the solver-free heuristics of Section~\ref{sec:agppc} can be scored identically. The weight $\lambda$ in \$/ton is a tie-breaking device to ensure planners can incorporate cost-service level tradeoff deliberately. In the experiments, we sweep it from $0$ to $100$ \$/ton to see how it affects the objective and the resulting strategy (Section~\ref{sec:exp_robust}).

\subsection{Sales, Stock, and Mass Balance Constraints}
This first group of constraints conserve the two stocks the plant carries, finished grades and the materials that make them, with the bulk feed as the one purchase the model does not choose. We thus have
\begin{align}
& z^{out}_{it} \leq d_{it}, && \forall i,t \label{eq:demand}\\
& z^{out}_{it} + \sigma^{e}_{it} = \sigma^{s}_{it} + z^{in}_{it}, && \forall i,t \label{eq:gradebal}\\
& \sigma^{s}_{it} = \begin{cases}\sigma^0_i & t=1\\ \sigma^{e}_{i,t-1} & t>1\end{cases}, && \forall i,t \label{eq:stocklink}\\
& \underline{\sigma}_i \leq \sigma^{e}_{it} \leq \bar{\sigma}_i, && \forall i,t \label{eq:stockbnd}\\
& z^{in}_{it} = \sum_{k} x_{ikt}, && \forall i,t \label{eq:inflow}\\
& w_{jkt} = \sum_{i} a_{ijk}\, x_{ikt}, && \forall j,k,t \label{eq:usage}\\
& \iota^{s}_{jt} + u_{jt} = \sum_{k} w_{jkt} + \iota^{e}_{jt}, && \forall j,t \label{eq:matbal}\\
& \sum_{k} w_{jkt} \le \iota^{s}_{jt} + u_{jt}, && \forall j,t \label{eq:matavail}\\
& \iota^{s}_{jt} = \begin{cases}\iota^0_j & t=1\\ \iota^{e}_{j,t-1} & t>1\end{cases}, && \forall j,t \label{eq:invlink}\\
& \underline{\iota}_j \leq \iota^{e}_{jt} \leq \bar{\iota}_j, && \forall j,t \label{eq:invbnd}\\
& u_{j^\ast t} = S_t, && \forall t \label{eq:bulk}\\
& u_{jt} \le M v_{jt}, && \forall j,t \label{eq:purchind}
\end{align}
Constraint~\eqref{eq:demand} caps sales at the forecast, and~\eqref{eq:gradebal} balances each grade over the period, splitting opening stock plus production inflow between sales and closing stock. Constraints~\eqref{eq:stocklink}-\eqref{eq:stockbnd} carry stock across periods and hold it within its warehouse bounds. Constraint~\eqref{eq:inflow} defines the inflow as the tonnage produced across all lines, and~\eqref{eq:usage} converts that tonnage into material draw through the bill of materials, which is indexed by process because the same grade consumes differently on different lines. Constraints~\eqref{eq:matbal}-\eqref{eq:matavail} do for materials what~\eqref{eq:gradebal} does for grades, the second stating explicitly that consumption cannot exceed what is on hand. Constraints~\eqref{eq:invlink}-\eqref{eq:invbnd} are the material counterparts of~\eqref{eq:stocklink}-\eqref{eq:stockbnd}. The feed is the exception, in that~\eqref{eq:bulk} fixes the purchase quantity of the bulk material $j^\ast$ at the contracted tonnage instead of leaving it to the model, while~\eqref{eq:purchind} switches on the purchase indicator of any material actually bought.

Note here that Constraint~\eqref{eq:bulk} encodes the take-or-pay contract; combined with the tight storage bound $\bar\iota_{j^\ast}$ in~\eqref{eq:invbnd} it forces in-month conversion. This pair of constraints is what turns the problem from a demand-chasing exercise into a capacity-allocation problem, i.e., the plant must convert roughly $S_t$ tons every month regardless of orders, so the decision should consider \emph{which} grades absorb that conversion.

\subsection{Capacity, Rate, Logic, and Transition Constraints}
This second group governs what each line can physically do inside a period, namely how many hours it has, how fast it may run, and which grades it is obliged to open. We have
\begin{align}
& \sum_{i} h_{ikt} \leq \bar{H}_k, && \forall k,t \label{eq:hourbudget}\\
& h_{ikt} \leq \bar{H}_k\, y_{ikt}, && \forall i,k,t \label{eq:hourlink}\\
& x_{ikt} \leq M y_{ikt}, && \forall i,k,t \label{eq:xlink}\\
& x_{ikt} \leq \bar{q}_{ik}, && \forall i,k,t \label{eq:xcap}\\
& \underline{r}_{ik}\, y_{ikt} \leq r_{ikt} \leq \bar{r}_{ik}\, y_{ikt}, && \forall i,k,t \label{eq:ratebnd}\\
& x_{i1t} \geq \underline{q}_i\, y_{i1t}, && \forall i,t \label{eq:minlot}\\
& \sum_{k} x_{ikt} + \sum_{i': \delta^{\min}_{ii'}>0 \,\vee\, \alpha_{ii'}>0}\; \sum_k x_{i'kt} \geq \underline{q}_i, && \forall i,t: \underline{q}_i > 0 \label{eq:minprod}
\end{align}
Constraint~\eqref{eq:hourbudget} caps the hours a line can operate in a period, and~\eqref{eq:hourlink} releases those hours to a grade only once it is switched on. Constraints~\eqref{eq:xlink}-\eqref{eq:xcap} do the same for tonnage, the first tying production to the on-off decision and the second capping it at the ceiling of the grade-line pair. Constraint~\eqref{eq:ratebnd} holds the rate inside the equipment envelope of that pair and drives it to zero when the pair is idle, while~\eqref{eq:minlot} requires any grade opened on the primary line to carry at least a technical minimum lot. Constraint~\eqref{eq:minprod} then obliges grades with a positive minimum production quantity to be produced each period, either directly or through their transition-linked or aggregated relatives.

Within a month the lines move between grades, and some transitions are admissible only after a minimum quantity of the successor has been produced, to flush the equipment. Write $\mathcal{S}(i) = \{i' : \delta^{\min}_{ii'} > 0\}$ for the set of successors of grade $i$ that carry such a requirement, and let $\mathcal{I}^{tr} = \{i : \mathcal{S}(i) \neq \emptyset\}$ be the grades that have one. We thus further have
\begin{align}
& \sum_{i' \in \mathcal{S}(i)} y^{tr}_{ii't} \;\geq\; y_{ikt}, && \forall i \in \mathcal{I}^{tr}, k, t \label{eq:trsel}\\
& \sum_{i' \in \mathcal{S}(i)} y^{tr}_{ii't} \;\leq\; \sum_{k} y_{ikt}, && \forall i \in \mathcal{I}^{tr}, t \label{eq:tract}\\
& \sum_{i} y^{tr}_{ii't}\, \delta^{\min}_{ii'} \;\leq\; \sum_{k} x_{i'kt}, && \forall i',t \label{eq:ytrans3}
\end{align}
Constraint~\eqref{eq:trsel} requires a grade that runs on any line to select at least one of its admissible successors, and~\eqref{eq:tract} permits a transition out of a grade only in periods in which that grade runs somewhere. Together they switch the transition indicators on exactly with the grades that drive them. Constraint~\eqref{eq:ytrans3} then enforces the minimum transition quantity on the successor's tonnage.

Finally, the central relationship ties quantity, rate, and hours,
\begin{align}
& x_{ikt} = r_{ikt}\, h_{ikt}, && \forall i,k,t. \label{eq:bilinear}
\end{align}
Note here that both rate and hours are decisions, in that the plant can run a grade slowly for long or quickly for short within the envelope of~\eqref{eq:ratebnd}, so this product cannot be avoided by fixing either factor. It is the first source of nonconvexity in the model, and Section~\ref{sec:mcc_method} shows how to relax it.
\subsection{Interprocess Compatibility Constraints}
\label{sec:compat}
The constraints in this subsection are the ones that distinguish the plant from a set of independent parallel machines, and they act at three levels: which grades may be on simultaneously, for how long, and at what rates. Running grade $i$ on P1 forces its compatible co-products online, and conversely a co-product can run only when a compatible anchor runs:
\begin{align}
& \Phi^{12}_{ii'}\, y_{i1t} \leq y_{i'2t}, \qquad \Phi^{13}_{ii'}\, y_{i1t} \leq y_{i'3t}, && \forall i,i',t \label{eq:force}\\
& M y_{i1t} \geq \sum_{i'} \Phi^{12}_{ii'}\, y_{i'2t}, \qquad M y_{i1t} \geq \sum_{i'} \Phi^{13}_{ii'}\, y_{i'3t}, && \forall i,t. \label{eq:sync}
\end{align}
Operating windows on the parallel lines must cover each other, hence
\begin{align}
& h_{i1t} \leq \sum_{i' \in \mathcal{C}_2(i)} h_{i'2t}, \qquad h_{i1t} \leq \sum_{i'' \in \mathcal{C}_3(i)} h_{i''3t}, && \forall i,t \label{eq:dur1}\\
& h_{i'2t} \leq \sum_{i: i' \in \mathcal{C}_2(i)} h_{i1t}, \qquad h_{i''3t} \leq \sum_{i: i'' \in \mathcal{C}_3(i)} h_{i1t}, && \forall i',i'',t. \label{eq:dur2}
\end{align}
Finally, the combined instantaneous rates must track the feed rate whenever the anchor operates. For every anchor $i$ and every compatible pair $(i',i'') \in \mathcal{C}_2(i) \times \mathcal{C}_3(i)$,
\begin{align}
& h_{i1t} \Big( r_{i1t} + \sum_{i' \in \mathcal{C}_2(i)} r_{i'2t} + \sum_{i'' \in \mathcal{C}_3(i)} r_{i''3t} \Big) \geq \rho_t\, h_{i1t}, && \forall i,t \label{eq:syncmin}\\
& h_{i1t} \left( r_{i1t} + r_{i'2t} + r_{i''3t} \right) \leq \rho_t\, h_{i1t}, && \forall i, (i',i''), t \label{eq:syncmax}
\end{align}
together with production-flow connectivity for grades with flow neighbors. Note that Constraints \eqref{eq:syncmin}-\eqref{eq:syncmax} are bilinear, hence together with~\eqref{eq:bilinear} they make the model a nonconvex mixed-integer quadratically-constrained program (MIQCP). We have two remarks on these constraints.

\begin{remark}[Forced co-production]
\label{rem:forced}
In the industrial data, $\bar{r}_{i1} + \bar{r}_{i'2} < \rho_t$ for every compatible pair $(i,i')$, i.e., P1 and P2 together cannot absorb the feed. Constraint~\eqref{eq:syncmin} therefore forces the granulation line to co-produce, at $r_{i''3t} \geq \underline{r}_{i''3} \geq 6$ tph or one fifth of the feed, whenever the primary line runs. Co-production is thus not a choice but a structural property of the plant, and it is the root cause of the limitation of the previous approach analyzed next.
\end{remark}

\begin{remark}[Aggregate time buckets]
\label{rem:buckets}
Constraints~\eqref{eq:force}-\eqref{eq:syncmax} enforce compatibility, window coverage, and rate tracking in aggregate monthly form. They are necessary conditions for a within-month schedule in which the three lines can be synchronized, but they do not fix the sequence itself: two grades whose monthly hours overlap in~\eqref{eq:dur1}-\eqref{eq:dur2} are not thereby required to run at the same instant. Consistent with hierarchical planning practice, within-month sequencing is delegated to the daily scheduling layer, which is outside the scope of this model. All plans reported in Section~\ref{sec:experiments} satisfy the aggregate conditions.
\end{remark}

\subsection{Perils of Margin-Based Prioritization}
\label{sec:whymislead}
With the model in place, we now explain why using the prevailing profitability metric can be dangerous, in the sense of carrying no economic guarantee and, more importantly, of risking an unnoticed and systematic destruction of value. To this end, we notice that practitioners' often rank grades by the single-product margin
\begin{equation}
\mathrm{SPM}_{i}(t) \;=\; p_{it} - c^{BOM}_{i1} - b_1,
\label{eq:spm}
\end{equation}
where $c^{BOM}_{ik} \;=\; \sum_{j} a_{ijk}\, c_j$, i.e. price minus direct cost per ton on the primary line, and fills the calendar from the top of that ranking. When lines are decoupled, \eqref{eq:spm} is a sensible proxy for the shadow value of capacity. However, in coupled plants, it fails for three reasons. First, committing P1 to grade $i$ forces co-production on P2 and P3 by Remark~\ref{rem:forced}, so the relevant economic object is the \emph{cluster} $\{i\} \cup \mathcal{C}_2(i) \cup \mathcal{C}_3(i)$, not just the grade. Second, the bottleneck is a shared \emph{hour} of coupled operation, not a ton of any single grade, so a correct metric must be expressed per hour. Third, forced co-products whose demand is exhausted are produced at pure cost, so the value of an anchor depends on the current demand state, which a static margin cannot reflect.

We support our claim above with the following result, constructed within the structure of Section~\ref{sec:problem} and with the plant's own rate envelope.

\begin{proposition}[Unbounded suboptimality of the margin ranking]
\label{prop:spm_bad}
For every $\Gamma > 0$ there is an instance satisfying Assumption~\ref{ass:fluid} on which the margin-ranking heuristic attains profit $\Pi_{\mathrm{SPM}} \le \Pi^\ast - \Gamma$, and $\Pi_{\mathrm{SPM}}$ is negative for $\Gamma$ large, while AGPPC-greedy attains $\Pi^\ast$.
\end{proposition}

\begin{proof}
Consider two anchors $A, B$ with disjoint clusters, one period, budget $\bar{H} = 1$ hour, feed rate $\rho = 30$ tph, and for both clusters anchor rate bounds $[15, 20]$ and a forced granule co-product with rate bounds $[6,10]$, mirroring the industrial data. Let anchor margins be $g_A = 105$ and $g_B = 100$ \$/ton, granule co-product margins $g_{A''} = -\gamma$ and $g_{B''} = +20$ \$/ton with parameter $\gamma > 0$; set $e = 0$, and let every member's demand satisfy $d_m \ge \bar{r}_m \bar{H} = \bar{r}_m$, so that Assumption~\ref{ass:fluid}(v) holds and Theorem~\ref{thm:greedy_opt} applies. The margin criterion ranks $A$ first ($105 > 100$) and, maximizing the anchor's own rate, runs $A$ at $(r_1, r_3) = (20, 10)$, giving $\Pi_{\mathrm{SPM}} = 105 \cdot 20 - \gamma \cdot 10 = 2100 - 10\gamma$. AGPPC evaluates both clusters: $\mathrm{AGPPC}_B = 100 \cdot 20 + 20 \cdot 10 = 2200 \ge \mathrm{AGPPC}_A$ for all $\gamma \ge -10$, so it runs $B$ and attains $\Pi^\ast = 2200$, optimal by Theorem~\ref{thm:greedy_opt}. The gap $\Pi^\ast - \Pi_{\mathrm{SPM}} = 100 + 10\gamma$ grows without bound in $\gamma$, and $\Pi_{\mathrm{SPM}} < 0$ for $\gamma > 210$. Choosing $\gamma \ge (\Gamma - 100)/10$ proves the claim. 
\end{proof}

Proposition~\ref{prop:spm_bad} has two implications. First, no instance-independent guarantee is available for the margin criterion, so any confidence in it must come from the instance itself instead of from the rule. Second, and more importantly, the loss is silent. The parameter $\gamma$ corresponds to a co-product whose demand has been exhausted; by the effective margin \eqref{eq:effmargin} defined in the next section its value is its negative direct cost, on the order of $-900$ \$/ton in our data, far beyond the threshold at which margin-based planning becomes value-destroying. Nothing in the planning workbook reports this, because the workbook prices the anchor and never prices what the anchor drags along with it. The construction is adversarial by design, so what it establishes is the absence of a worst-case guarantee and not a claim about typical magnitudes. In Section~\ref{sec:exp_metric}, we show the empirical counterpart on the industrial instance and in Section~\ref{sec:exp_heur}, we expose the profit consequences.

\section{The AGPPC Framework}
\label{sec:agppc}
We propose the AGPPC framework to overcome the peril presented in Section~\ref{sec:whymislead}. We first give an overview of the approach and our target guarantees (Section~\ref{sec:overview_guar}). We then derive the economic object that should replace the margin, namely the coupled-hour value of a production cluster, and the allocation rule built on it (Sections~\ref{sec:columns}-\ref{sec:greedy_alg}). Finally we establish what that rule guarantees (Section~\ref{sec:theory}).

\subsection{Overview and Target Guarantees}
\label{sec:overview_guar}
Our framework comprises two layers. The first prices one hour of coupled operation at the current demand state, by reading the smallest committable unit off the compatibility constraints and valuing it. The second allocates the coupled-hour budget greedily by that price. The key to obtaining a guarantee lies in what the price measures, which requires two properties:
\begin{enumerate}
\item \emph{Commitment consistency}: the priced object must be exactly what the plant can commit to, i.e., a compatible triple of grades together with a rate mix that tracks the feed, so that no plan built from these prices can be infeasible for the coupling constraints 
\item \emph{State awareness}: the price must fall when a cluster member's remaining demand is exhausted, so that forced output which can only go to stock at cost is charged to its anchor.
\end{enumerate}
The first property is what makes the allocation feasible by construction; the second is what makes it economically correct as the plan is built. Together they give a rate of value creation per coupled hour, and the guarantee we target is that allocating the budget by that rate recovers the optimum of the planning problem, once the combinatorial side constraints that the metric deliberately ignores are relaxed. We call the relaxation in which this holds the \emph{fluid relaxation}, defined in Section~\ref{sec:theory}. From a risk quantification viewpoint, we can assert that because the relaxation is a linear program, its optimum is computable on any instance, so the ratio between the allocated profit and that optimum bounds the loss even when the hypotheses of the guarantee fail. This is the analogue, in a planning context, of shifting the target from an exact value to a bound that can be estimated correctly.

\subsection{Production Clusters and Co-Production Columns}
\label{sec:columns}
For an anchor grade $i$ producible on P1, define its \emph{cluster}
\begin{equation}
    \mathcal{G}(i) = \{(i,1)\} \cup \{(i',2): i' \in \mathcal{C}_2(i)\} \cup \{(i'',3): i'' \in \mathcal{C}_3(i)\}.
\end{equation}
By \eqref{eq:syncmax}, at any instant one co-product from $\mathcal{C}_2(i)$ and one from $\mathcal{C}_3(i)$ operate alongside the anchor, and by \eqref{eq:syncmin}-\eqref{eq:syncmax} their rates satisfy $r_1 + r_2 + r_3 = \rho_t$ within the machine bounds. We call a triple $c = (i, i', i'')$ with $i' \in \mathcal{C}_2(i)$ and $i'' \in \mathcal{C}_3(i)$ a \emph{co-production column}, and the set of rate vectors feasible for it is the polytope
\begin{equation}
\mathcal{R}_c(t) \;=\; \Big\{ (r_1,r_2,r_3) \;:\; r_1+r_2+r_3 = \rho_t, \;\; \underline{r}_{mk} \le r_m \le \bar{r}_{mk} \; \forall (m,k) \in c \Big\}.
\label{eq:column_poly}
\end{equation}
A column is thus the smallest unit of operation the model permits. The MIQCP never chooses a grade in isolation: it chooses a column and a point of $\mathcal{R}_c(t)$. In that view, our metric below simply prices that choice.

\subsection{The AGPPC Metric}
\label{sec:agppc_def}
Let 
\begin{equation}
g_{mk}(t) = p_{mt} - c^{BOM}_{mk} - b_k    
\end{equation}
 denote the unit margin of member $(m,k)$ and 
 \begin{equation}
     e(c) = c^E \sum_{k \in \mathrm{procs}(c)} e_k
 \end{equation}
 the electricity cost of one hour of coupled operation. Given the remaining demand vector $D = (D_1, \dots, D_I)$ at the current planning state, define the \emph{effective margin}
\begin{equation}
\tilde{g}_{mk}(t; D) \;=\;
\begin{cases}
g_{mk}(t) & D_m > 0 \quad \text{(output is sellable)}\\[2pt]
-\,c^{BOM}_{mk} - b_k & D_m = 0 \quad \text{(forced output goes to stock at cost)} .
\end{cases}
\label{eq:effmargin}
\end{equation}
The AGPPC of anchor $i$ is then the best achievable profit per hour of coupled operation over all of its columns and feasible rate mixes:
\begin{equation}
\mathrm{AGPPC}_{i}(t; D) \;=\; \max_{(i',i'')} \; \max_{r \in \mathcal{R}_{(i,i',i'')}(t)} \;\; \sum_{(m,k)} \tilde{g}_{mk}(t; D)\, r_{m} \;-\; e(i,i',i'').
\label{eq:agppc}
\end{equation}
The inner problem is a three-variable linear program over a simplex slice of a box, solvable in closed form or by any LP solver in microseconds, and the outer maximization enumerates the $|\mathcal{C}_2(i)| \times |\mathcal{C}_3(i)|$ co-product pairs. AGPPC has units of dollars per coupled hour, which is the shadow-price scale of the true bottleneck, and it is demand-state aware through \eqref{eq:effmargin}. We note that when clusters are trivial ($\mathcal{C}_2 = \mathcal{C}_3 = \emptyset$ and $\bar r_{i1} \geq \rho_t$), AGPPC reduces to $\mathrm{SPM}_i(t) \cdot \rho_t - e$, so that it generalizes the margin criterion into more complex coupled plants.

\subsection{A Demand-Aware Greedy Allocation}
\label{sec:greedy_alg}
With AGPPC, planners can rank and devise their plans. In Algorithm~\ref{alg:greedy}, we allocate the coupled bottleneck, namely the $\bar{H}_t = \min\{\bar{H}_1, S_t / \rho_t\}$ hours of feed conversion available per period, greedily by AGPPC. After each allocation the remaining demands and hence the effective margins change, so the metric is re-evaluated, and each step runs until the next demand-saturation breakpoint. The margin-based counterpart, which represents current practice in our experiments, differs in two places. First, its anchors are ranked by \eqref{eq:spm}, and second, the rate mix within a column maximizes the anchor's own rate instead of the cluster value in AGPPC.

\begin{algorithm}[htbp]
\caption{AGPPC-greedy allocation of coupled hours (one period $t$)}
\label{alg:greedy}
\begin{algorithmic}[1]
\State $\bar{H} \gets \min\{\bar{H}_1,\, S_t/\rho_t\}$; \quad $D_i \gets \max\{0,\, d_{it} - \sigma^{s}_{it}\}$ for all $i$
\While{$\bar{H} > 0$}
    \State for each anchor $i$: compute $\mathrm{AGPPC}_i(t; D)$, its optimal column $c_i^\ast$ and mix $r^\ast_i$ \Comment{Eq.~\eqref{eq:agppc}}
    \State $i^\ast \gets \arg\max_i \mathrm{AGPPC}_i(t; D)$
    \If{$\mathrm{AGPPC}_{i^\ast} \le 0$} \textbf{break} \Comment{no profitable coupled hour remains}
    \EndIf
    \State $\Delta \gets \min\Big\{\bar{H},\; \min_{m:\, D_m > 0} D_m \big/ \textstyle\sum_{k} r^\ast_{i^\ast,(m,k)}\Big\}$ \Comment{run to next saturation breakpoint}
    \State commit $(i^\ast, c^\ast_{i^\ast}, r^\ast_{i^\ast})$ for $\Delta$ hours; update $x, h$; $D_m \gets \max\{0, D_m - \Delta \sum_k r^\ast_{(m,k)}\}$; $\bar{H} \gets \bar{H} - \Delta$
\EndWhile
\State sell $z^{out}_{it} = \min\{d_{it},\, \sigma^{s}_{it} + \sum_k x_{ikt}\}$; carry ending stocks to $t{+}1$
\end{algorithmic}
\end{algorithm}

\subsection{Guarantees}
\label{sec:theory}
The AGPPC metric and its resulting algorithm were derived from the coupling structure of the model. The natural question is how much of the model's optimum they capture. We answer it on a fluid relaxation of the single-period problem that retains the coupling constraints, which drive the phenomenon under study, while relaxing the combinatorial side constraints. To that end, we first introduce the main assumption and definition.

\begin{assumption}
\label{ass:fluid}
(i) A single period with bottleneck budget $\bar{H} = \min\{\bar{H}_1, S_t/\rho_t\}$; (ii) demands act as sales caps and unsold output has zero value but full cost; (iii) minimum production quantities, transitions, production-flow connectivity, and stock and inventory side constraints are relaxed; (iv) prices and costs are linear; and (v) \emph{(non-saturation)} every cluster member satisfies $d_m \ge \bar{r}_{mk}\,\bar{H}$, so that no member's remaining demand can reach zero within the budget.
\end{assumption}

\begin{definition}[Fluid relaxation $\mathrm{P_F}$]
\label{def:fluid}
Over columns $c \in \mathcal{C}$ (all feasible co-production triples), choose hours $h_c \geq 0$, member production $q^c_{mk}$, and sales $s^c_{mk}$ to solve
\begin{align}
\Pi_F^\ast = \max \;\; & \sum_{c} \Big[ \sum_{(m,k) \in c} \big( p_m s^c_{mk} - (c^{BOM}_{mk}+b_k)\, q^c_{mk} \big) - e(c)\, h_c \Big] \label{eq:fluid_obj}\\
\text{s.t.} \;\;
& \underline{r}_{mk} h_c \leq q^c_{mk} \leq \bar{r}_{mk} h_c, \qquad s^c_{mk} \le q^c_{mk}, && \forall c, (m,k) \in c \nonumber\\
& \textstyle\sum_{(m,k) \in c} q^c_{mk} = \rho_t\, h_c, && \forall c \nonumber\\
& \textstyle\sum_c h_c \leq \bar{H}, \qquad \textstyle\sum_{c}\sum_{k} s^c_{ik} \leq d_i, && \forall i. \nonumber
\end{align}
\end{definition}

Comparing this relaxed model against Section~\ref{sec:model}, $\mathrm{P_F}$ keeps the rate bounds, the feed-tracking condition in per-column form, the hour budget, and the demand caps. It drops \eqref{eq:minprod}-\eqref{eq:ytrans3} and the inventory bounds, resulting in a linear program, and any plan produced by Algorithm~\ref{alg:greedy} is feasible for it, so $\Pi_F^\ast$ upper-bounds the heuristic's profit. For anchor $a$, let $V_a(h; d)$ denote the optimal value of $\mathrm{P_F}$ restricted to the columns of anchor $a$ with hour budget $h$ and demand caps $d$. We have the following result.

\begin{lemma}
\label{lem:concave}
Fix an anchor $a$. (a) $V_a(\cdot\,; d)$ is piecewise-linear and concave on $[0,\bar{H}]$. (b) Under Assumption~\ref{ass:fluid}, including (v), it is linear: $V_a(h; d) = \mathrm{AGPPC}_a(t; d)\, h$ for every $h \in [0, \bar{H}]$. (c) In general, $\mathrm{AGPPC}_a(t; D)$ evaluated at remaining demand $D$ is the rate of profit of continuing at the instantaneously best column and mix, and it upper-bounds the right derivative of $V_a$ at that state, with equality while no member of the cluster is saturated.
\end{lemma}

\begin{proof}
(a) $V_a(h; d)$ is the optimal-value function of a linear program in which $h$ enters only the right-hand side of the budget constraint; such value functions are piecewise-linear and concave. (b) Under (v), any plan that allocates $h \le \bar{H}$ hours to anchor $a$ produces at most $\bar{r}_{mk} h \le \bar{r}_{mk} \bar{H} \le d_m$ tons of each member, so every demand cap of $\mathrm{P_F}$ restricted to $a$ is slack and effective margins~\eqref{eq:effmargin} coincide with true margins throughout. The restricted program then reduces to choosing, at each instant, a column $c$ of $a$ and a mix $r \in \mathcal{R}_c(t)$, with instantaneous profit $\sum_{(m,k)} g_{mk}(t) r_m - e(c) \le \mathrm{AGPPC}_a(t;d)$ by \eqref{eq:agppc}; holding the maximizing pair for $h$ hours is feasible and attains the bound. (c) The first claim is~\eqref{eq:agppc} evaluated at $D$. For the inequality, note that the continuation value ignores a cost the relaxation can avoid: consuming a member's remaining demand at the instantaneously best rate leaves that member producing at its lower rate bound $\underline{r}_{mk}$, at pure cost, for the remainder of the run, whereas $\mathrm{P_F}$ may throttle the member from the outset so that its demand lasts the whole run. 
\end{proof}

\noindent Lemma~\ref{lem:concave} is the bridge between the metric and the model. AGPPC is not an ad-hoc score but the rate at which a coupled hour creates value in the relaxed model at the current demand state. Part (c) also marks the boundary of that correspondence, and the following theorem states the regime in which it is exact.

\begin{theorem}[Optimality under non-saturating demand]
\label{thm:greedy_opt}
Suppose Assumption~\ref{ass:fluid} holds, including (v). Then Algorithm~\ref{alg:greedy} returns an optimal solution of the fluid relaxation, $\Pi_{\mathrm{greedy}} = \Pi_F^\ast$. 
\end{theorem}

\begin{proof}
By Lemma~\ref{lem:concave}(b) every demand cap of $\mathrm{P_F}$ is slack at every feasible point, so the caps may be deleted without changing $\Pi_F^\ast$. The demand caps are the only constraints of $\mathrm{P_F}$ that couple distinct anchors; once they are gone the program separates, and by Lemma~\ref{lem:concave}(b) each anchor's contribution is linear in its allocation. Hence
\begin{equation*}
    \Pi_F^\ast = \max \Big\{ \sum_a \mathrm{AGPPC}_a(t;d)\, h_a \;:\; \sum_a h_a \leq \bar{H},\; h_a \geq 0 \Big\},
\end{equation*}
a continuous knapsack with a single budget, solved by placing the entire budget on any anchor of maximal AGPPC, and by allocating nothing if that maximum is nonpositive. This is exactly what Algorithm~\ref{alg:greedy} does: under (v) the metric is constant along the allocation, no saturation breakpoint is reached, and the loop commits the whole budget to $\arg\max_a \mathrm{AGPPC}_a(t;d)$, stopping early only when the maximum is nonpositive. The greedy allocation is feasible for $\mathrm{P_F}$ by construction, hence optimal. Because the coupling between anchors was removed by (v) rather than by an assumption on cluster structure, disjointness plays no role. 
\end{proof}

Theorem~\ref{thm:greedy_opt} both generalizes and narrows the usual disjointness argument: it dispenses with any condition on cluster structure, but it holds only where demands are rich relative to the coupled-hour budget. However, it could break when \emph{demands saturate}. Once a member's remaining demand reaches zero its forced output is made at pure cost, and Algorithm~\ref{alg:greedy}, which maximizes the instantaneous rate of profit, spends scarce demand at the highest admissible rate instead of spreading it across the run; by Lemma~\ref{lem:concave}(c) the realized value then falls short of $V_a$. Also, its effectiveness is reduced when \emph{clusters overlap}. In our example, the granule co-products are shared across several anchors in the industrial instance, so the demand caps couple anchors even after each anchor's own value function is understood. Section~\ref{sec:exp_theory} separates the two effects on controlled instances to complete our analysis. The next result addresses the side constraint most likely to bind in practice.

\begin{theorem}[Additive guarantee under minimum-run constraints]
\label{thm:additive}
Let $\mathrm{P_F^{min}}$ be $\mathrm{P_F}$ with semi-continuous hours, i.e., each opened anchor must receive at least $\underline{h}_a = \underline{q}_a / \bar{r}_{a1}$ hours. Let $\Pi^\ast_{\min}$ be its optimum and let Algorithm~\ref{alg:greedy} be run with the same restriction, skipping anchors whose remaining budget cannot cover $\underline{h}_a$ and truncating the final allocation. Then, under Assumption~\ref{ass:fluid},
\[
\Pi_{\mathrm{greedy}} \;\geq\; \Pi^\ast_{\min} \;-\; \max_a\, \mathrm{AGPPC}^0_a \cdot \underline{h}_a,
\]
where $\mathrm{AGPPC}^0_a$ denotes the metric evaluated at full demand.
\end{theorem}

\begin{proof}
Since $\mathrm{P_F}$ relaxes $\mathrm{P_F^{min}}$, $\Pi^\ast_{\min} \leq \Pi_F^\ast$. Run the unrestricted greedy of Theorem~\ref{thm:greedy_opt}, which attains $\Pi_F^\ast$. Its allocation satisfies the minimum-run restriction at every opened anchor except possibly the last one opened before budget exhaustion, whose allocation may fall below $\underline{h}_a$. Closing that anchor produces a feasible allocation for $\mathrm{P_F^{min}}$ and forfeits at most the value of the truncated run. By concavity (Lemma~\ref{lem:concave}) the per-hour value of any run never exceeds its at-opening marginal value, which is at most $\mathrm{AGPPC}^0_a$, so the forfeited value is at most $\mathrm{AGPPC}^0_a\, \underline{h}_a \le \max_a \mathrm{AGPPC}^0_a\, \underline{h}_a$. Therefore $\Pi_{\mathrm{greedy}} \geq \Pi_F^\ast - \max_a \mathrm{AGPPC}^0_a \underline{h}_a \geq \Pi^\ast_{\min} - \max_a \mathrm{AGPPC}^0_a \underline{h}_a$. 
\end{proof}

The additive term of Theorem~\ref{thm:additive} is small in practice. In the industrial data, $\underline{h}_a \le \underline{q}_a / \bar r_{a1} \approx 50/20 = 2.5$ hours against a monthly budget of roughly $700$ hours, so the additive loss is below $0.5\%$ of horizon profit and the minimum-lot rule is essentially free for the heuristic. Furthermore, while the two guarantees above are \emph{a priori} (i.e., they hold on instances meeting their hypotheses), the next result is their \emph{a posteriori} counterpart.

\begin{proposition}[Per-instance optimality certificate]
\label{prop:cert}
Let an instance satisfy Assumption~\ref{ass:fluid}(i)-(iv), with clusters possibly overlapping and demands possibly saturating. Every plan returned by Algorithm~\ref{alg:greedy} is feasible for $\mathrm{P_F}$, hence $\Pi_{\mathrm{greedy}} \le \Pi_F^\ast$. Since $\mathrm{P_F}$ is a linear program, $\Pi_F^\ast$ is computable in polynomial time, and the ratio $\Pi_{\mathrm{greedy}} / \Pi_F^\ast$ is a valid optimality certificate for that instance, requiring neither disjointness nor Assumption~\ref{ass:fluid}(v).
\end{proposition}

\begin{proof}
Algorithm~\ref{alg:greedy} only ever commits a column $c$ at a point of $\mathcal{R}_c(t)$, never exceeds the budget $\bar{H}$, and never sells more than remaining demand; these are exactly the constraints of $\mathrm{P_F}$, so the returned plan is feasible for it. A feasible point of a maximization program is bounded by its optimum, and $\mathrm{P_F}$ is a linear program by Definition~\ref{def:fluid}. 
\end{proof}

\begin{corollary}[Checkable condition for exactness]
\label{cor:check}
If Algorithm~\ref{alg:greedy} exhausts the hour budget without any member's remaining demand reaching zero, then $\Pi_{\mathrm{greedy}} = \Pi_F^\ast$. 
\end{corollary}

\begin{proof}
Write $\mathrm{AGPPC}^0_a = \mathrm{AGPPC}_a(t;d)$ for the metric at full demand. For any anchor $a$ and any $h \le \bar{H}$ we have $V_a(h;d) \le \mathrm{AGPPC}^0_a\, h$, since the objective of $\mathrm{P_F}$ restricted to $a$ is the time integral of $\sum_{(m,k)} \tilde g_{mk} r_m - e(c)$ and effective margins never exceed true margins. Relaxing the shared demand caps to per-anchor caps can only increase the optimum and makes it separable, so $\Pi_F^\ast \le \max\{\sum_a V_a(h_a;d) : \sum_a h_a \le \bar H\} \le (\max_a \mathrm{AGPPC}^0_a)\, \bar{H}$. Under the stated condition the effective margins never change, so the algorithm runs a single anchor attaining $\mathrm{AGPPC}^0_a$ throughout and spends the whole budget, giving $\Pi_{\mathrm{greedy}} = (\max_a \mathrm{AGPPC}^0_a)\bar{H} \ge \Pi_F^\ast$. Proposition~\ref{prop:cert} supplies the reverse inequality. 
\end{proof}

The most important aspect of Proposition~\ref{prop:cert} is we do not assume non-saturation, and we do not assume that clusters are disjoint. If the instance happens to satisfy those hypotheses, Theorem~\ref{thm:greedy_opt} applies and the certificate holds. If it does not, the certificate still returns a number and Corollary~\ref{cor:check} lets the algorithm report which of the two cases it is in from its own trace. In this sense Proposition~\ref{prop:cert} makes our framework robust: it gives a tight assessment when the demand state is favorable, and a correct, though conservative, one when it is not.

\section{Solution Methodology}
\label{sec:method}
\subsection{Exact Model and Solvers}
The exact model \eqref{eq:objective}-\eqref{eq:syncmax} is a nonconvex MIQCP. We implement it in Pyomo \citep{hart2011pyomo} and solve it by spatial branch-and-bound, using Gurobi with \texttt{NonConvex=2} \citep{gurobi2024} when a full license is available and SCIP \citep{bolusani2024scip} otherwise, both consuming the identical model through our open-source pipeline. For the 38-grade industrial instance the exact model has roughly 2{,}300 variables, about 1{,}800 of them binary, and 5{,}200 constraints per period.

\subsection{Linearized MILP Baseline}
\label{sec:mcc_method}
The two sources of nonconvexity identified in Section~\ref{sec:model}, namely the rate-hour product \eqref{eq:bilinear} and the rate-tracking constraints \eqref{eq:syncmin}-\eqref{eq:syncmax}, can both be replaced by linear inequalities built from the natural bounds of the participating variables. The result is a MILP that a practitioner without advanced solver can deploy.

\begin{proposition}[McCormick linearization: bounds and error]
\label{prop:mcc}
Let $\mathrm{P_{MC}}$ be the MILP obtained from the exact model by (i) replacing each bilinear constraint \eqref{eq:bilinear} with its McCormick envelope over $r_{ikt} \in [\underline{r}_{ik} y_{ikt},\, \bar{r}_{ik} y_{ikt}]$ and $h_{ikt} \in [0, \bar{H}_k y_{ikt}]$,
\begin{align}
x_{ikt} &\ge \underline{r}_{ik} h_{ikt}, &
x_{ikt} &\ge \bar{r}_{ik} h_{ikt} + \bar{H}_k r_{ikt} - \bar{r}_{ik} \bar{H}_k y_{ikt}, \nonumber\\
x_{ikt} &\le \bar{r}_{ik} h_{ikt}, &
x_{ikt} &\le \underline{r}_{ik} h_{ikt} + \bar{H}_k r_{ikt} - \underline{r}_{ik} \bar{H}_k y_{ikt} + M(1 - y_{ikt}),
\label{eq:mccormick}
\end{align}
and (ii) replacing the bilinear synchronization constraints \eqref{eq:syncmin}-\eqref{eq:syncmax} with indicator (big-M on $y_{i1t}$) counterparts. Then (a) some optimal solution of the exact model is feasible for $\mathrm{P_{MC}}$ with equal objective, so $\Pi_{\mathrm{MC}} \geq \Pi^\ast_{\mathrm{exact}}$; (b) the pointwise envelope error is at most $(\bar{r}_{ik} - \underline{r}_{ik})\, \bar{H}_k / 4$ per term; and (c) fixing $\hat{y}$ from a $\mathrm{P_{MC}}$ solution and re-optimizing the continuous exact model yields a feasible plan with $\Pi_{\mathrm{repair}} \leq \Pi^\ast_{\mathrm{exact}} \leq \Pi_{\mathrm{MC}}$, i.e., a computable optimality sandwich.
\end{proposition}

\begin{proof}
(a) The four inequalities \eqref{eq:mccormick} are the convex and concave envelopes of $x = rh$ on the box $[\underline{r}, \bar{r}] \times [0, \bar{H}]$ \citep{mccormick1976computability, alkhayyal1983jointly}; every point with $x = rh$ satisfies them, and when $y = 0$ all variables collapse to zero in both models. The indicator form of \eqref{eq:syncmin}-\eqref{eq:syncmax} is implied by the bilinear form whenever $h_{i1t} > 0$, and an exact-optimal solution can always be normalized so that $y_{i1t}=1 \Rightarrow h_{i1t} > 0$, since closing idle lines never reduces the objective; on that support the two forms coincide. (b) The maximal deviation between $rh$ and its envelopes over a box is $(\bar{r}-\underline{r})(\bar{H}-0)/4$, attained at the box center \citep{alkhayyal1983jointly}. (c) With $\hat{y}$ fixed, the continuous model is the exact feasible set to that support, so its optimum is feasible for the exact model and hence a lower bound. 
\end{proof}


\subsection{Operationalizing the Benchmark Strategies}
\label{sec:strategies}
We compare three strategies: make-to-order (MTO), make-to-stock (MTS), and hybrid. All three strategies share the full constraint set and differ only in policy overlays, so that observed differences isolate the strategy itself. Under \textbf{MTO}, which is current practice, production is triggered by current-period orders, i.e., discretionary pellet production may not exceed current demand up to the technical minimum lot, $x_{i1t} + x_{i2t} \leq \max\{d_{it},\, \underline{q}_i\}$, with P3 output exempt because coupling forces it (Remark~\ref{rem:forced}). Under \textbf{MTS}, sales are served entirely from on-hand stock, $z^{out}_{it} \leq \sigma^{s}_{it}$, so that every ton sold was produced in an earlier period. The proposed \textbf{hybrid} imposes no overlay, letting the optimizer choose per grade and period where between the two poles to operate.

\section{Numerical Experiments and Results}
\label{sec:experiments}
In this section we present six experiments on the industrial instance: the metric comparison, the strategy comparison, the linearized baseline, the heuristics against the optimizer, a direct test of the guarantees, and robustness and scalability, closing with the tool deployed at the partner facility. These experiments are chosen to test each claim of the framework in turn.  The first isolates the ranking produced by the metric from any downstream planning decision, while the second and third measure what the strategy space and the linearization are worth once a solver is in the loop. The fourth attributes profit to the metric alone by holding every other rule fixed. The fifth evaluates the certificate of Proposition~\ref{prop:cert} on the plant as it stands and then, on controlled instances, separates the two hypotheses of Theorem~\ref{thm:greedy_opt} to identify which of them the residual loss should be charged to. The last two delimit the range over which the results travel. All experiments were run on an eight-core workstation with models built in Pyomo 6.x. Unless stated otherwise, MILPs are solved with HiGHS to a 1\% gap, exact MIQCPs with Gurobi 11 (\texttt{NonConvex=2}) to a 1\% gap or a 900\,s limit, with SCIP 9 reproducing all runs without a commercial license through the released scripts, and heuristics run in pure Python and SciPy.

\subsection{Industrial Case Study}
\label{sec:casestudy}
The case study uses operational data from a petrochemical-derived polymer facility in Indonesia: 38 product grades (pellet grades on P1 and P2, granule grades, and treated variants), 45 materials (30 raw materials including the bulk propylene-derived feed, and 15 utilities), and 3 coupled processes. Monthly demand forecasts, contract prices, bills of material, rate bounds, compatibility matrices, transition minimums, and the supply contract ($S_t \approx 20{,}675$-$22{,}675$ tons/month at $\rho_t \approx 30.5$ tph) are taken from the facility's 2025 planning workbook. The main horizon is January to March 2025 ($T=3$), and scalability experiments extend to $T=11$. Total demand over the main horizon is about 66{,}400 tons against a conversion capability of roughly the same magnitude, so the plant is capacity-tight and prioritization genuinely matters.

Product and process names are anonymized throughout (grades G01-G38, processes P1-P3). Figure~\ref{fig:compat_structure} shows the compatibility structure that drives the coupling: each of the 20 anchor grades producible on P1 admits only a small set of co-products on P2 (65 of $20 \times 27$ possible pairs) and on P3 (24 of $20 \times 10$), so committing P1 to an anchor pins the plant to a narrow cluster of simultaneous products.

\begin{figure}[htbp]
    \centering
    \includegraphics[width=0.8\linewidth]{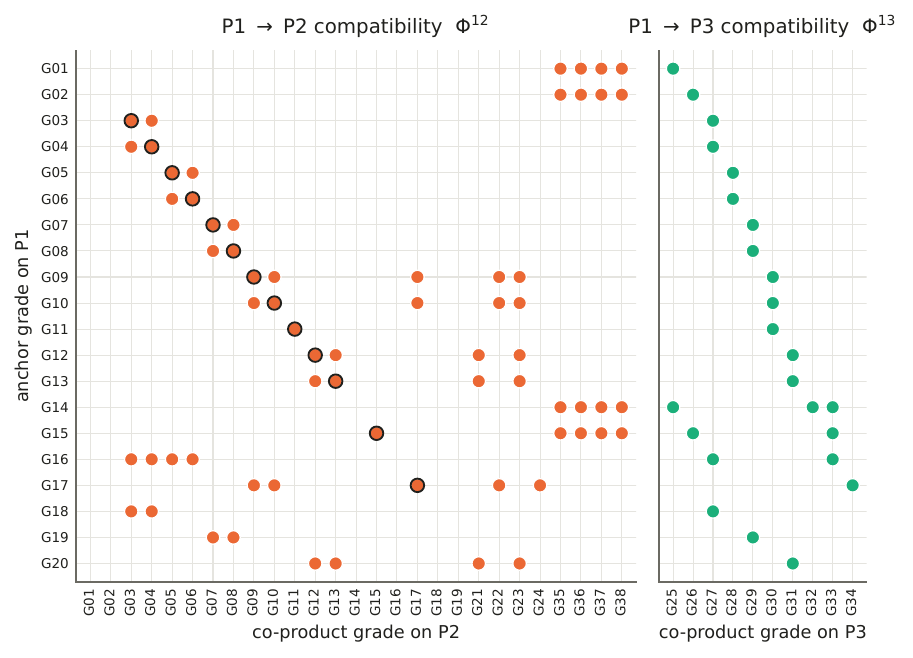}
    \caption{Interprocess compatibility structure of the industrial instance. Left: P1 anchor grades (rows) versus admissible co-products on P2 ($\Phi^{12}$); ringed dots mark self-compatible anchors, where the same grade runs on both pelletizing lines. Right: admissible co-products on the granulation line P3 ($\Phi^{13}$).}
    \label{fig:compat_structure}
\end{figure}

\subsection{Experiment 1: The Metric versus Single-Product Margins}
\label{sec:exp_metric}
\input{tables/E1_metric}
Table~\ref{tab:metric} ranks all 20 anchor grades under both criteria for January 2025. $\Delta$ means SPM rank minus AGPPC rank. The two rankings correlate only moderately (Spearman $\rho = 0.71$, $p < 0.001$; Kendall $\tau = 0.54$). The grade ranked ninth by margin (G15) is the \emph{top} cluster by AGPPC, because its cluster co-produces the same grade on P2 while running a high-value granule, whereas the two grades the margin criterion would prioritize first (G14, G17) anchor clusters worth 43\% to 46\% less per coupled hour. Grades G20 and G10 similarly jump from ranks 11 and 13 to 7 and 8. A planner filling the calendar by the left ranking therefore commits the plant's scarcest resource in nearly the reverse order for several most consequential choices.

\subsection{Experiment 2: Strategy Comparison and the Linearized Baseline}
\label{sec:exp_strategy}
\input{tables/E2_strategy}
Table~\ref{tab:strategy} compares the three strategies of Section~\ref{sec:strategies} on the January to March horizon, each solved with the exact MIQCP, the McCormick MILP bound, and the MILP-plus-repair feasible plan. Here, we report objective (M\$), usage-based operating profit $\Pi$ (M\$), service level (\%), utilization (\%), solver gap (\%), and wall time (s). McCormick rows are relaxation bounds (Proposition~\ref{prop:mcc}) and are not physically realizable plans, which is why they report higher output at lower utilization. Repair rows are feasible plans on the MILP's binary support. \emph{Service level} is served tons as a fraction of total demanded tons; \emph{utilization} is operated hours as a fraction of the hour budget of all lines. Objective values follow~\eqref{eq:objective} and $\Pi$ follows~\eqref{eq:profit}, i.e., the two differ in level only through the treatment of the take-or-pay purchase, and rank the strategies identically.

We note that the hybrid strategy dominates. MTO caps discretionary production at current orders, so when the take-or-pay obligation exceeds order-driven pellet demand the model cannot pre-build next month's high-value sales and coupled hours are pushed into less valuable clusters; operating profit falls by nearly half relative to the hybrid (\$4.30M versus \$8.31M over the quarter) at a lower service level (93.7\% versus 100\%). MTS is far worse. With sales served exclusively from previously produced stock, first-period sales are limited to opening inventories while the feed obligation still forces full-scale production, so the strategy front-loads cost and back-loads or forfeits revenue, and its operating profit over the truncated horizon is negative. However, this is because first-period sales are capped by opening stock and the horizon carries no salvage value on terminal inventory, part of the loss is a cold-start artifact of a three-month window rather than a property of make-to-stock operation as such. We retain it because it brackets the strategy space from one side while the strict order-driven variants of Section~\ref{sec:strategies} bracket it from the other; the managerially relevant comparison is hybrid versus MTO. The hybrid uses build-ahead selectively, mostly for grades whose clusters are cheap to run when the obligation exceeds current demand, which is exactly the flexibility that the coupled structure rewards.

\subsection{Experiment 3: Heuristics versus Optimization}
\label{sec:exp_heur}
\input{tables/E3_heuristics}
Table~\ref{tab:heuristics} evaluates the two greedy planners of Section~\ref{sec:agppc} against the exact optimization on the three-month horizon, all scored by the same usage-based operating profit. Both heuristics honor the take-or-pay conversion floor, in that after the demand-driven phase the remaining obligated feed is converted by the least-costly available cluster. AGPPC-greedy reaches \$6.99M, within 16\% of the optimizer's \$8.31M, at a 97.6\% service level, and improves on the margin-based planner by 7.6\% in profit and two service points. The comparison is designed to isolate the contribution of the metric: the two planners share every rule except how anchors are ranked and how the rate mix inside a column is chosen, so the 7.6\% is attributable to the metric alone.

\subsection{Experiment 4: Testing the Guarantees}
\label{sec:exp_theory}
\input{tables/E5_theory}
Table~\ref{tab:theory} evaluates the certificate of Proposition~\ref{prop:cert} in the idealized single-period setting of Assumption~\ref{ass:fluid}(i)-(iv). The fluid relaxation is built from the plant's own compatibility matrices, giving 82 co-production columns with granule co-products shared across anchors, so no disjointification is performed anywhere in this experiment and the ratio is a valid certificate on the instance as it stands. AGPPC-greedy attains 95.8\%, 95.6\% and 95.8\% of $\Pi_F^\ast$ in January, February and March, so that cluster overlap and demand saturation together cost the heuristic at most 4.4\% of the fluid optimum on this plant. The margin-based planner is not merely worse but value-destroying in this setting, at $-812\%$ to $-1113\%$ of the fluid optimum, because with a single period and no inventory credit every forced co-product whose demand it has exhausted is produced at full cost. 

\input{tables/E5b_validation}

We also explore the natural question which of the hypotheses contributes the residual 4.4\% should be charged to. Table~\ref{tab:theory_validation} shows 60 random 12-anchor problems per cell, with the plant's rate envelopes, crossing Assumption~\ref{ass:fluid}(v) against cluster structure. Where (v) holds the greedy is exactly optimal in every single instance, whether clusters are pairwise disjoint or every anchor draws its co-products from one common pool, which is the content of Theorem~\ref{thm:greedy_opt} and confirms that overlap alone is harmless. Where (v) fails the greedy loses even with disjoint clusters (0.52\% on average, 3.82\% at worst), and overlap then compounds the loss (2.17\% on average, 33.59\% at worst).

\subsection{Experiment 5: Robustness and Scalability}
\label{sec:exp_robust}
\input{tables/E4_scalability}
Finally, we report how far our results extend beyond the data from our case study. Scaling all demands by $\pm 10\%$ and $\pm 20\%$ and perturbing the feed rate $\rho_t$ by $\pm 5\%$ and $\pm 10\%$ gives nine scenarios including the baseline. AGPPC-greedy leads the margin-based planner in eight of them, by \$0.08M to \$3.06M of quarterly operating profit against a baseline lead of \$0.49M. In the ninth, a 10\% reduction in feed rate, the two are within 1.2\% of each other and the margin-based planner is marginally ahead, because a slower feed leaves enough coupled hours that anchor choice stops being scarce. The advantage is therefore state-dependent. It is largest when capacity is tight relative to demand and shrinks toward zero when it is not. The certificate of Proposition~\ref{prop:cert} is far more stable than the advantage it bounds, staying between 94.7\% and 96.0\% across all nine scenarios.

The penalty weight $\lambda$ deserves the same sensitivity analysis. We begin at 1 \$/ton to 100 \$/ton, sweeping  it over $\{0, 0.1, 1, 5, 10, 50, 100\}$ on the linearized model leaves the objective, the operating profit, the service level, and the entire optimal binary support unchanged for all three strategies. On the exact MIQCP, $\lambda = 0$ and $\lambda = 1$ give operating profits of \$8{,}313{,}554 and \$8{,}313{,}420 on the same support of 330 open grade-line-period pairs.

\label{sec:exp_scal}
Furthermore, Table~\ref{tab:scalability} reports model sizes and solve times as the horizon grows from 1 to 11 months. The McCormick MILP scales gracefully, staying under two seconds even at $T=11$ with 25{,}500 variables. The exact MIQCP solves to a 1\% gap within seconds for the horizons of practical interest for monthly S\&OP ($T \le 6$), but at $T=11$ it exhausts the 900\,s budget with a 9.2\% residual gap, which is exactly the regime where the MILP bound and the repair step earn their keep, while the greedy heuristic is horizon-insensitive. This supports a tiered deployment: AGPPC-greedy for interactive what-if planning, the MILP-plus-repair pipeline for long-horizon or frequent re-planning with certified gaps, and the exact model for the monthly S\&OP cycle.

\subsection{The Deployed Planning Tool}
\label{sec:exp_tool}
The model is deployed at the partner facility through a web application serving the Pyomo model core used in the experiments. To use it, planners upload the monthly workbook, select the demand basis, and inspect the optimized plan through interactive visualizations. Figures~\ref{fig:app_fulfillment} and~\ref{fig:app_heatmap} reproduce the tool's own outputs on the anonymized January instance under its two demand bases: order-driven planning, where demand equals confirmed sales orders and which corresponds to the MTO strategy, and forecast-driven planning, where demand equals the sales forecast and which corresponds to the MTS strategy.

\begin{figure}[htbp]
    \centering
    \includegraphics[width=0.78\linewidth]{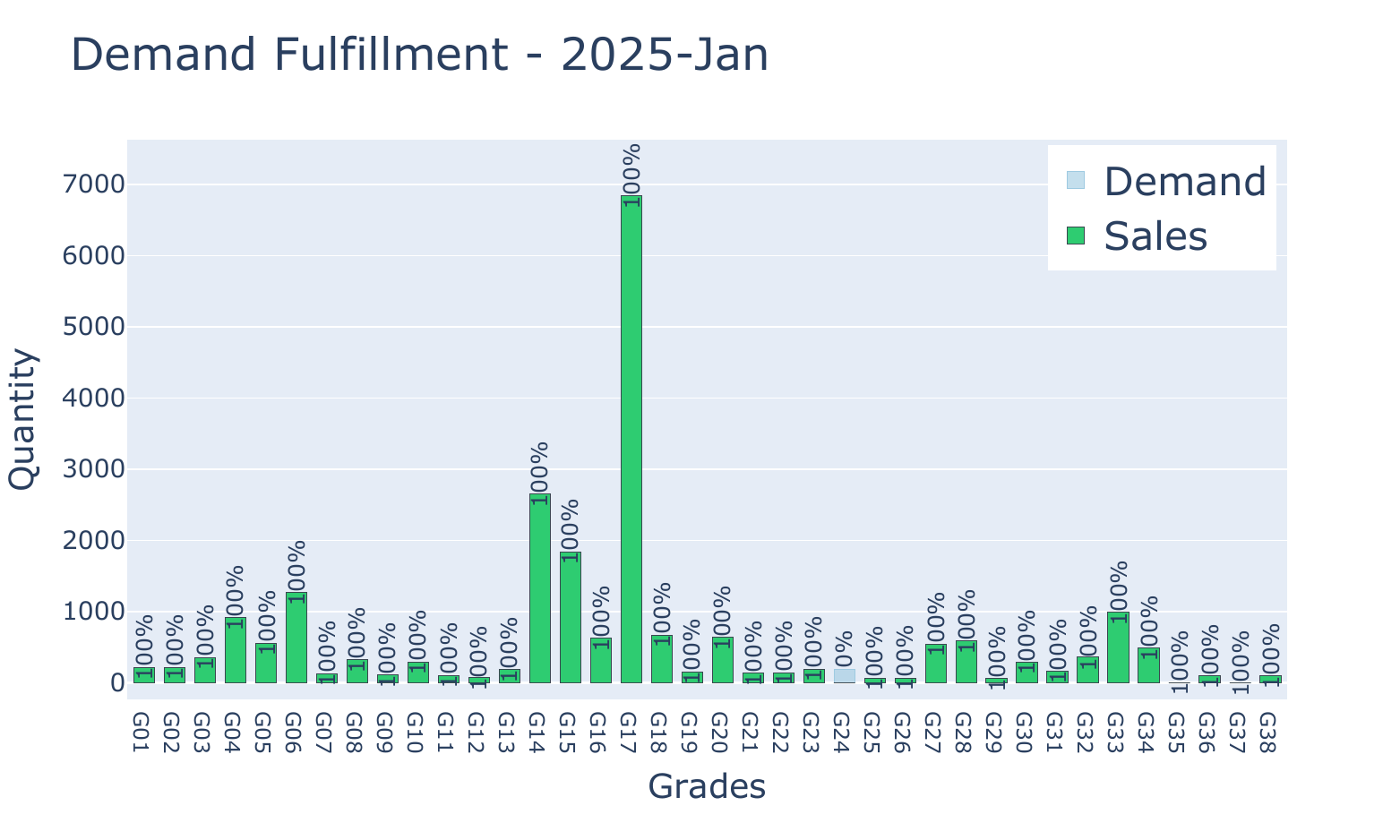}\\[1pt]
    \includegraphics[width=0.78\linewidth]{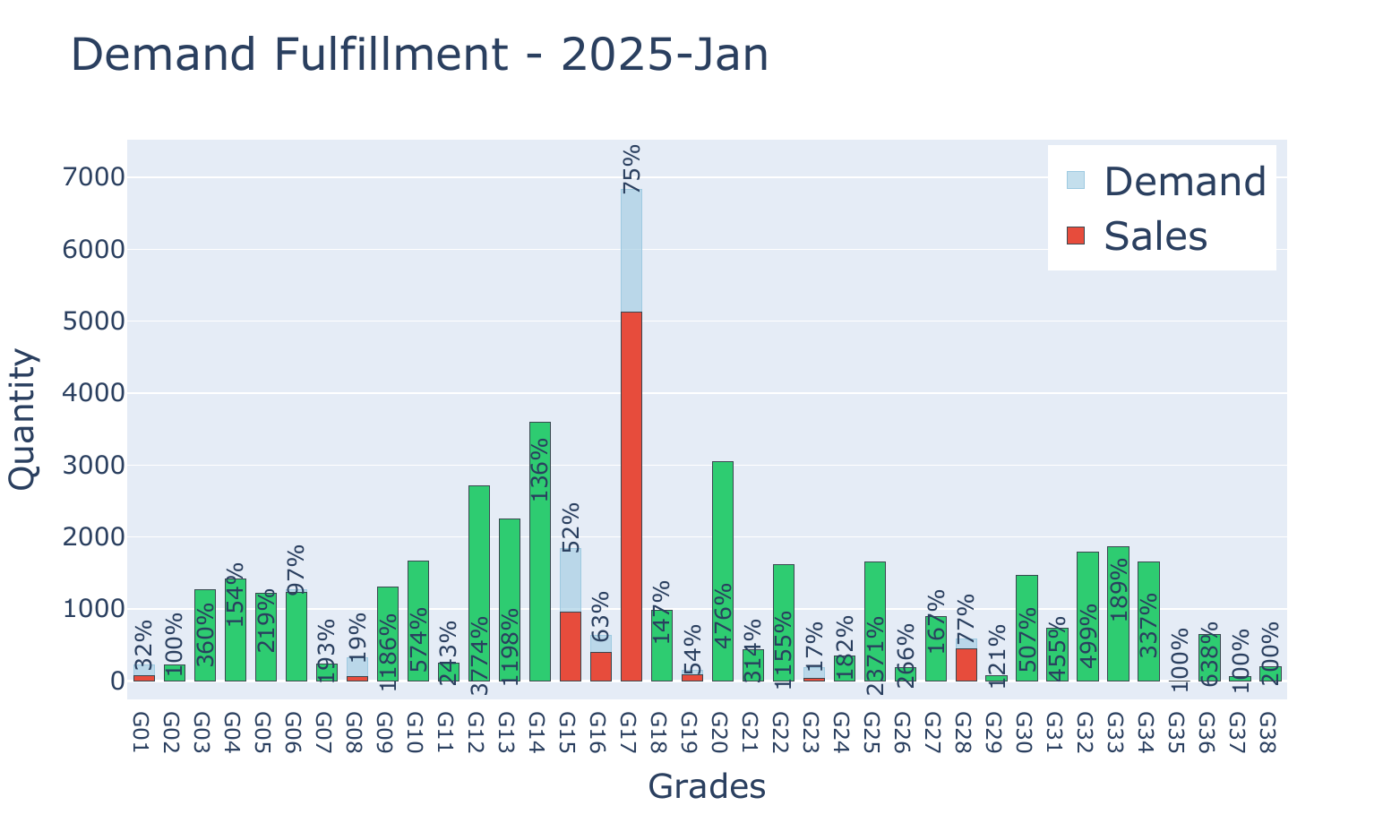}
    \caption{Demand-fulfillment view of the deployed tool on the anonymized January instance. Top: order-driven basis (MTO strategy), where every confirmed order is served in full. Bottom: forecast-driven basis (MTS strategy), where grades receiving the forced surplus conversion sell above forecast while others are left short (red bars).}
    \label{fig:app_fulfillment}
\end{figure}

Furthermore, the production-target heatmap of Figure~\ref{fig:app_heatmap} shows the plant-level operations planning. Grade-line assignments concentrate on a small set of anchor clusters, and the granulation line's assignment follows the anchor selected on the primary pelletizing line through the compatibility sets. The tool also shows the plan's material logistics, through a flow view tracing procurement into material pools and on into grade tonnage, and through a per-grade balance of stock, production, and sales across the quarter, combining all S\&OP decisions, shown in Figure~\ref{fig:app_sankey}.

\begin{figure}[htbp]
    \centering
    \includegraphics[width=0.72\linewidth]{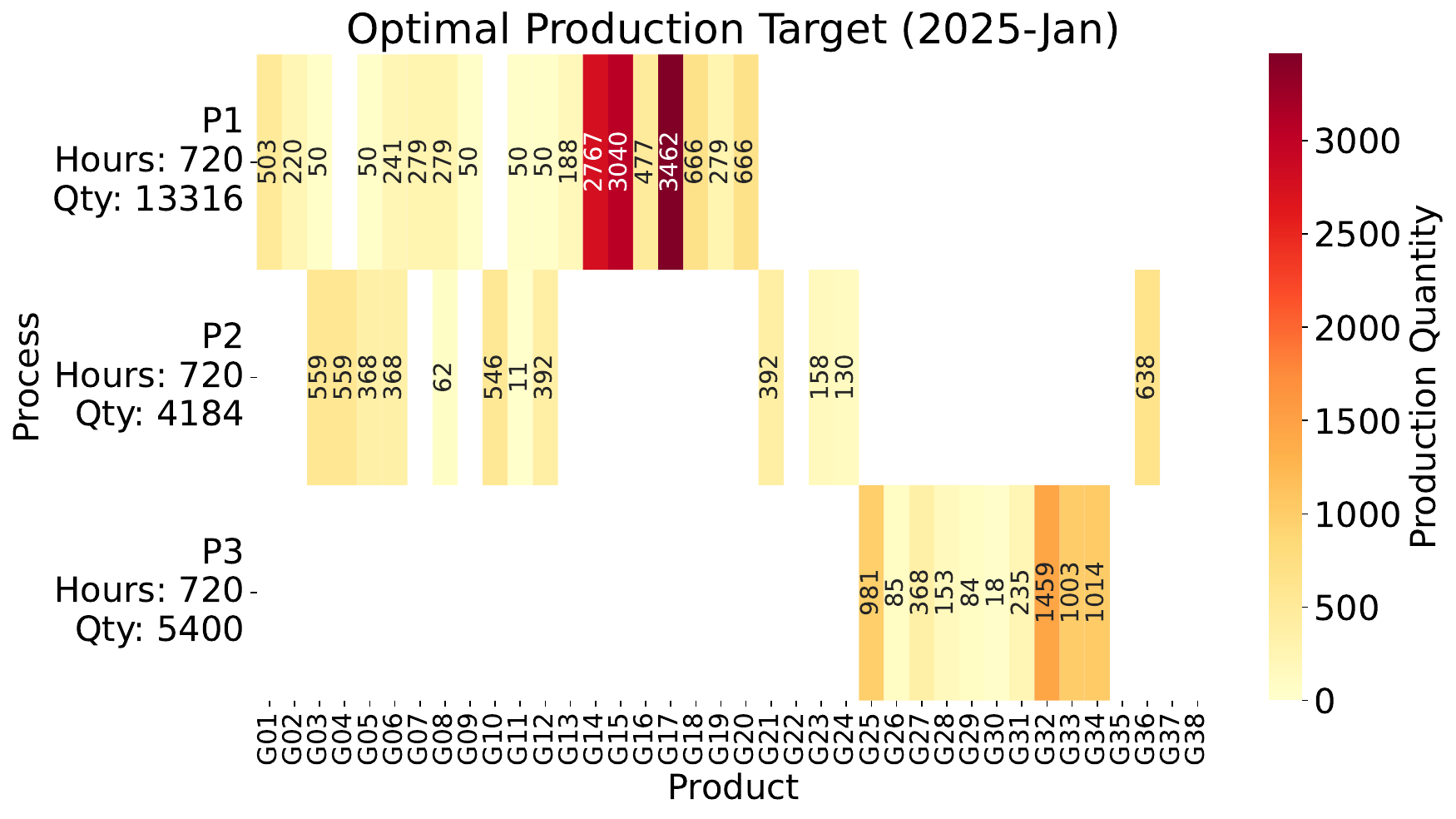}\\[1pt]
    \includegraphics[width=0.72\linewidth]{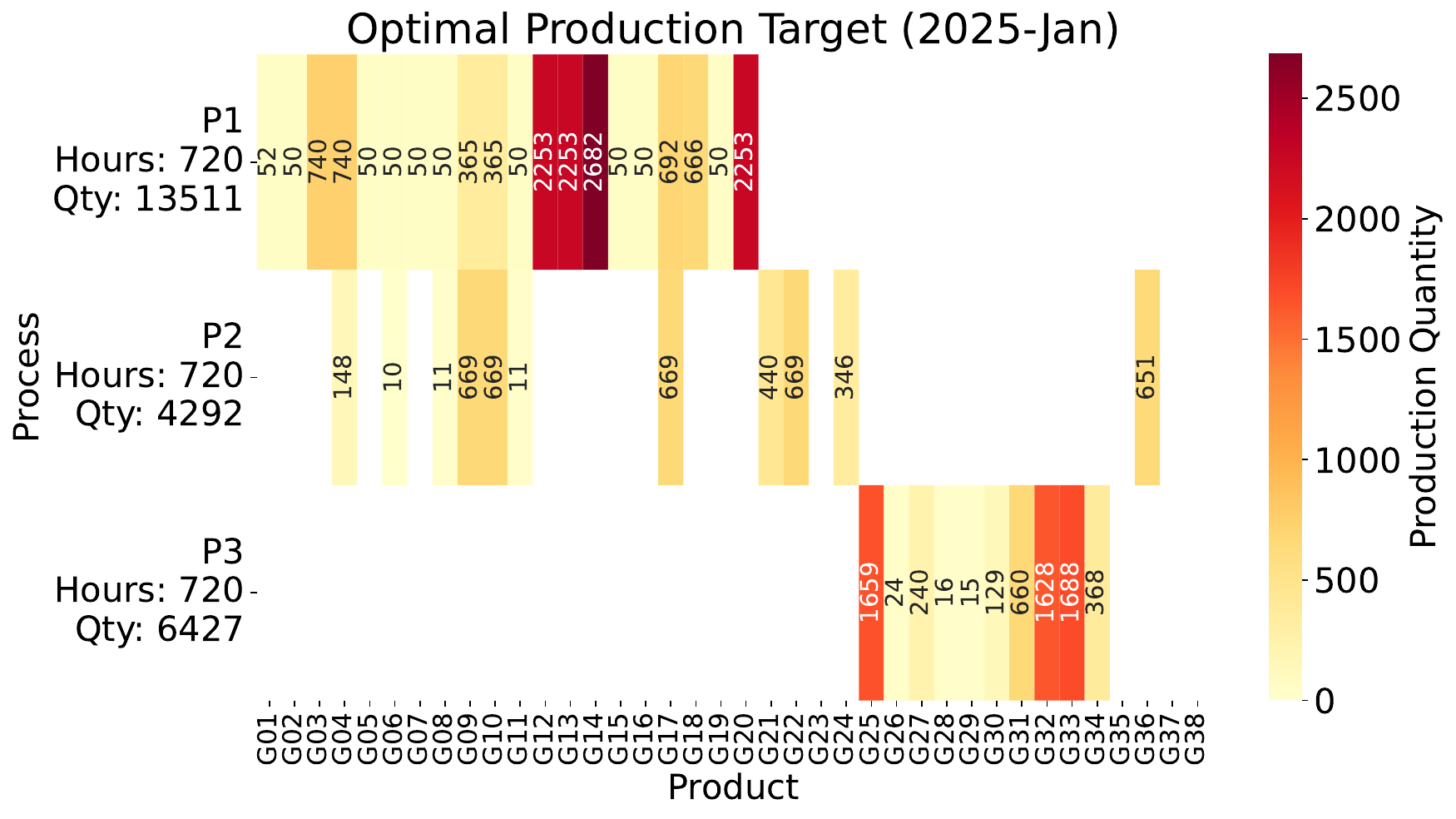}
    \caption{Production-target heatmaps from the deployed tool for January (top: order-driven; bottom: forecast-driven)}
    \label{fig:app_heatmap}
\end{figure}

\begin{figure}[htbp]
    \centering
    \includegraphics[width=0.9\linewidth]{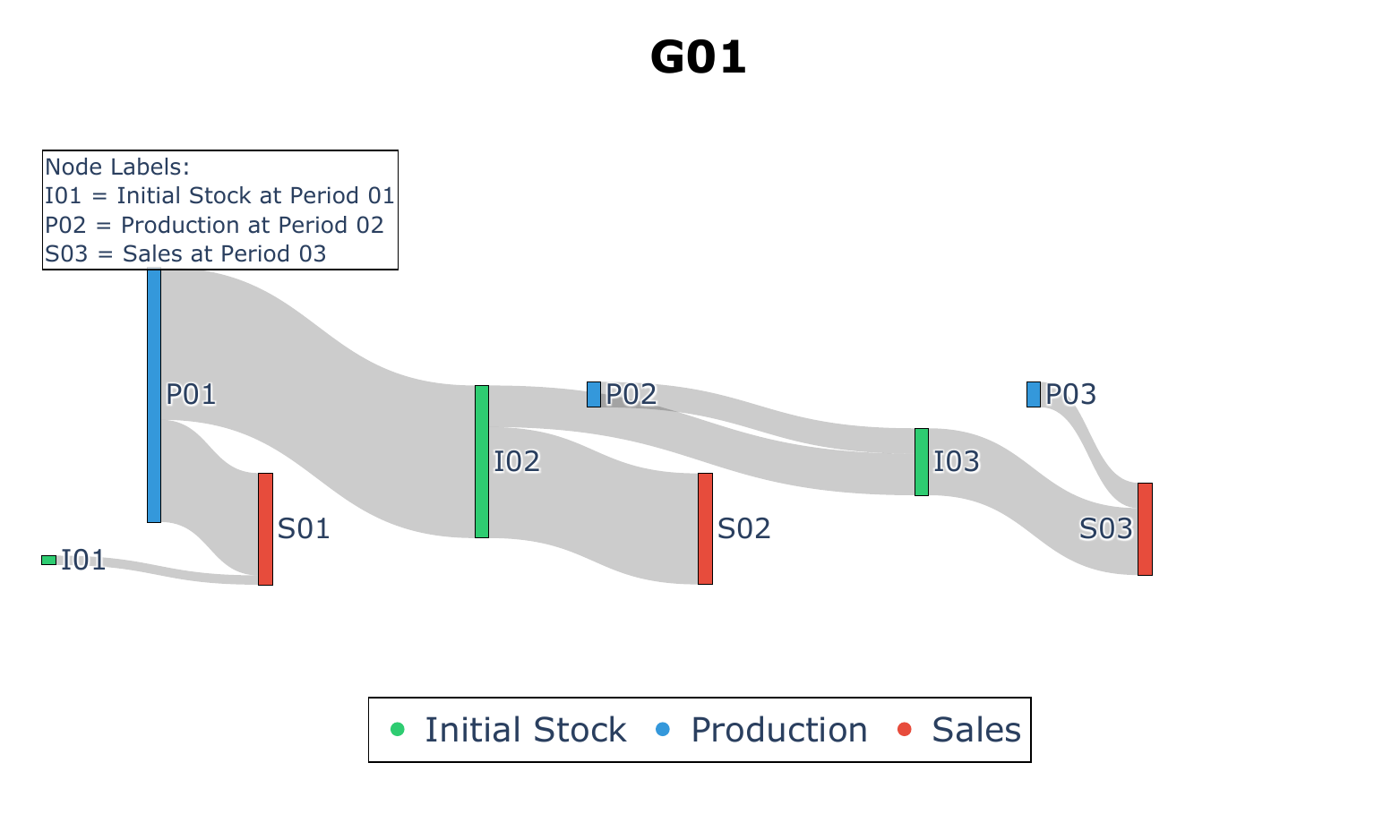}
     \caption{Product grades logistics flow showing inventory, production, and sales for Grade G01.}
    \label{fig:app_sankey}
\end{figure}

\section{Discussion and Future Work}
\label{sec:conclusion}

We briefly discuss three main insights from our study. First, \emph{price the bottleneck, not the product}. In coupled continuous plants the scarce resource is an hour of coupled operation, so any per-ton, per-product metric, however carefully costed, would likely yield suboptimal results. A better approach is highlighted in our second insights, which is \emph{co-production discipline dominates portfolio selection}. The advantage of coupling-aware planning comes overwhelmingly from avoiding forced co-production of demand-exhausted grades instead of from picking different headline products. This means planners should track remaining \emph{cluster} demand and not only product demand. Here, the advantage is state-dependent, i.e. it is modest though still material when the take-or-pay obligation leaves little discretion and more decisive whenever discretionary capacity exists. Finally, \emph{hybrid beats pure strategies for structural reasons}. The feed contract forces conversion even when current demand is weak, so the option to build ahead selectively must coexist with demand-triggered production. 

We close by discussing some limitations of our work and offer prospective follow-up works. First, the quantitative results are specific to the instance class we studied, and the perturbation experiments might have delimited them. While the structural findings seem to hold for any plant satisfying Remark~\ref{rem:forced}, the size of the metric's advantage depends on how tight coupled hours are relative to demand. Second, the residual gap of Algorithm~\ref{alg:greedy} is now localized to demand saturation instead of to cluster structure, so an allocation rule that prices the \emph{duration} of a cluster's demand instead of only its instantaneous rate would most likely yield better results. Proving a guarantee for such a rule under saturation is the natural next theoretical step. Third, AGPPC is single-period and myopic in inventory, which is why it cannot price the build-ahead option that the hybrid strategy exploits. Adding a continuation value on carried stock to the effective margin~\eqref{eq:effmargin} would address this, but we defer it to future works. 

\section{Conclusion}

In this paper we proposed a coupling-aware approach to sales and operations planning in process plants whose parallel lines are tied together by a shared bulk feed and grade compatibility. We argue that the prioritization criterion in standard use, the single-product gross profit margin, encounters a potential lack of economic validity in such plants and, moreover, the peril of an undetected and systematic destruction of value. In contrast, our proposed approach prices the object the plant can actually commit to, namely a co-production column and a feasible rate mix, by combining the cluster concept with demand-aware effective margins, and allocates the coupled-hour budget by that price. Our approach is designed to strengthen, rather than to replace, the criterion already in place: it reduces to the familiar margin ranking exactly when the plant is decoupled, and it is computed from the same data planners already maintain. Our guarantees on correctness rely on a fluid relaxation that retains precisely the coupling constraints, and they are exhibited in relation to the demand state, with more conservativeness when demands saturate within the coupled-hour budget. We view this work as a first step toward a rigorous economics of coupled production planning, and we envision it to lay the foundation for further improvements in prioritizing complex production systems.

\section*{Data and Code Availability}
\noindent The model implementation, the anonymized dataset, and experiment scripts are available in the companion open-source repository \url{https://github.com/ai-vnv/CoupledPlantOPT}.

\section*{Declaration of Competing Interest}
\noindent The authors declare that they have no known competing financial interests or
personal relationships that could have appeared to influence the work reported
in this paper.

\section*{Declaration of Generative AI Usage}
\noindent During the preparation of this work the authors used Claude Chat (Opus 5 and Fable) in order to
refine words, provide critical assessments, and find overlooked references. After
using this tool, the authors reviewed and edited the content as needed and take
full responsibility for the content of the published article. The authors also used Copilot to polish and document the companion code for better legibility and reproducibility of the results. 

\bibliographystyle{elsarticle-num-names}
\bibliography{reference}

\end{document}

%% file: tables/E1_metric.tex
\begin{table}[htbp]
\centering
\caption{Anchor grades ranked by the traditional SPM versus the proposed coupling-aware metric AGPPC, January 2025. }
\label{tab:metric}
\begin{tabular}{@{}crcrcc@{}}
\toprule
\textbf{Anchor grade} & \textbf{SPM (\$/ton)} & \textbf{SPM rank} & \textbf{AGPPC (\$/h)} & \textbf{AGPPC rank} & $\Delta$ \\
\midrule
G15 & 115.7 & 9 & 6,473 & 1 & +8 \\
G17 & 138.2 & 2 & 3,659 & 2 & +0 \\
G14 & 138.5 & 1 & 3,507 & 3 & -2 \\
G09 & 118.3 & 5 & 3,290 & 4 & +1 \\
G16 & 121.6 & 3 & 3,205 & 5 & -2 \\
G11 & 120.6 & 4 & 3,186 & 6 & -2 \\
G20 & 113.3 & 11 & 3,173 & 7 & +4 \\
G10 & 110.3 & 13 & 3,169 & 8 & +5 \\
G12 & 105.1 & 16 & 3,041 & 9 & +7 \\
G07 & 116.7 & 8 & 2,991 & 10 & -2 \\
G13 & 101.2 & 18 & 2,983 & 11 & +7 \\
G08 & 114.4 & 10 & 2,953 & 12 & -2 \\
G19 & 108.1 & 14 & 2,846 & 13 & +1 \\
G06 & 117.4 & 6 & 2,781 & 14 & -8 \\
G18 & 117.0 & 7 & 2,737 & 15 & -8 \\
G02 & 106.9 & 15 & 2,732 & 16 & -1 \\
G05 & 112.2 & 12 & 2,694 & 17 & -5 \\
G01 & 85.7 & 20 & 2,622 & 18 & +2 \\
G03 & 101.8 & 17 & 2,464 & 19 & -2 \\
G04 & 97.6 & 19 & 2,392 & 20 & -1 \\
\bottomrule
\end{tabular}
\end{table}

%% file: tables/E2_strategy.tex
\begin{table}[htbp]
\centering
\caption{Strategy comparison on the industrial instance (Jan--Mar 2025, $T=3$).}
\label{tab:strategy}
\begin{tabular}{@{}llrrrrrr@{}}
\toprule
\textbf{Strategy} & \textbf{Model} & \textbf{Obj.} & $\Pi$ & \textbf{Serv.} & \textbf{Util.} & \textbf{Gap} & \textbf{Time} \\
\midrule
Hybrid (proposed) & Exact MIQCP & 5.45 & 8.31 & 100.0 & 88.0 & 1.0 & 0.6 \\
Hybrid (proposed) & McCormick MILP (bound) & 5.56 & 8.42 & 100.0 & 80.3 & 0.1 & 0.4 \\
Hybrid (proposed) & MILP + repair (feasible) & 5.45 & 8.31 & 100.0 & 88.0 & 0.9 & 0.7 \\
\addlinespace
MTO & Exact MIQCP & 1.43 & 4.30 & 93.7 & 89.9 & 1.0 & 3.1 \\
MTO & McCormick MILP (bound) & 1.61 & 4.48 & 93.7 & 75.6 & 0.0 & 0.4 \\
MTO & MILP + repair (feasible) & 1.43 & 4.30 & 93.7 & 89.6 & 1.0 & 3.1 \\
\addlinespace
MTS & Exact MIQCP & -4.79 & -1.91 & 83.6 & 88.8 & 1.0 & 158 \\
MTS & McCormick MILP (bound) & -3.94 & -1.06 & 84.8 & 78.4 & 0.0 & 0.3 \\
MTS & MILP + repair (feasible) & -4.79 & -1.91 & 83.6 & 88.7 & 1.5 & 900 \\
\addlinespace
\bottomrule
\end{tabular}
\end{table}

%% file: tables/E3_heuristics.tex
\begin{table}[htbp]
\centering
\caption{Heuristics versus optimization on the industrial instance (Jan--Mar 2025).}
\label{tab:heuristics}
\begin{tabular}{@{}lrrrr@{}}
\toprule
\textbf{Planner} & $\Pi$ (M\$) & \textbf{Service (\%)} & \textbf{Production (t)} & \textbf{Time (s)} \\
\midrule
Exact hybrid optimization & 8.31 & 100.0 & 62,019 & 0.6 \\
AGPPC-greedy (proposed) & 6.99 & 97.6 & 62,509 & 1.1 \\
SPM-greedy (margin practice) & 6.50 & 95.9 & 61,778 & 0.8 \\
\bottomrule
\end{tabular}
\end{table}

%% file: tables/E5_theory.tex
\begin{table}[htbp]
\centering
\caption{Per-instance certificate (Proposition~\ref{prop:cert}) in the idealized single-period setting.}
\label{tab:theory}
\begin{tabular}{@{}lrrrrr@{}}
\toprule
 & \multicolumn{1}{c}{\textbf{Fluid LP} $\Pi_F^\ast$} & \multicolumn{2}{c}{\textbf{AGPPC-greedy}} & \multicolumn{2}{c}{\textbf{SPM-greedy}} \\
\cmidrule(lr){2-2}\cmidrule(lr){3-4}\cmidrule(lr){5-6}
\textbf{Month} & (M\$) & (M\$) & \% of $\Pi_F^\ast$ & (M\$) & \% of $\Pi_F^\ast$ \\
\midrule
January  & 1.765 & 1.691 & 95.8 & $-15.390$ & $-871.7$ \\
February & 1.764 & 1.685 & 95.6 & $-14.323$ & $-812.1$ \\
March    & 1.440 & 1.379 & 95.8 & $-16.021$ & $-1112.9$ \\
\bottomrule
\end{tabular}
\end{table}

%% file: tables/E5b_validation.tex
\begin{table}[htbp]
\centering
\caption{Gap of Algorithm~\ref{alg:greedy} to the fluid optimum}
\label{tab:theory_validation}
\begin{tabular}{@{}llrrrr@{}}
\toprule
\textbf{Assumption~\ref{ass:fluid}(v)} & \textbf{Clusters} & \textbf{Mean} & \textbf{Median} & \textbf{90th pct.} & \textbf{Max} \\
\midrule
Holds & disjoint & 0.000 & 0.000 & 0.000 & 0.000 \\
Holds & overlapping & 0.000 & 0.000 & 0.000 & 0.000 \\
\addlinespace
Violated & disjoint & 0.52 & 0.14 & 1.49 & 3.82 \\
Violated & overlapping & 2.17 & 0.00 & 3.29 & 33.59 \\
\bottomrule
\end{tabular}
\end{table}

%% file: tables/E4_scalability.tex
\begin{table}[htbp]
\centering
\caption{Scalability with the planning horizon $T$ (hybrid strategy).}
\label{tab:scalability}
\begin{tabular}{@{}rrrrrrr@{}}
\toprule
$T$ & \textbf{Variables} & \textbf{MILP time} & \textbf{MILP obj.} & \textbf{Exact time} & \textbf{Exact obj.} & \textbf{Exact gap} \\
\midrule
1 & 2,322 & 0.1 & 2.86 & 0.2 & 2.81 & 0.8 \\
2 & 4,644 & 0.3 & 7.16 & 0.4 & 7.10 & 0.2 \\
3 & 6,966 & 0.4 & 5.56 & 0.9 & 5.45 & 1.0 \\
6 & 13,932 & 0.8 & 9.34 & 2.0 & 9.15 & 0.6 \\
11 & 25,542 & 1.6 & 28.40 & 903 & 23.40 & 9.2 \\
\bottomrule
\end{tabular}
\end{table}